\documentclass[a4paper,12pt]{article}

\usepackage{amsmath}
\usepackage{amsthm}
\usepackage{enumerate}
\usepackage{amssymb,amsfonts,latexsym,mathtools}
\usepackage{mathrsfs}
\usepackage{bm}
\usepackage{cases}
\usepackage{color}

\newcommand{\Levy}{L\'{e}vy}

\newcommand{\R}{\mathbb{R}}
\newcommand{\F}{\mathscr{F}}

\newcommand{\e}{\varepsilon}

\newcommand{\cadlag}{c\`{a}dl\`{a}g}
\renewcommand{\P}{\mathbb{P}}

\numberwithin{equation}{section}

\makeatletter
\renewcommand\section{\@startsection {section}{1}{\z@}%
{-3.5ex \@plus -1ex \@minus -.2ex}%
{2.3ex \@plus.2ex}%
{\normalfont\large\bf}}
\makeatother

\makeatletter
\renewcommand\subsection{\@startsection {subsection}{1}{\z@}%
{-3.5ex \@plus -1ex \@minus -.2ex}%
{2.3ex \@plus.2ex}%
{\normalfont\normalsize\bf}}
\makeatother

\theoremstyle{plain}
\newtheorem{thm}{Theorem}[section]
\newtheorem{lem}[thm]{Lemma}

\newtheorem{prop}[thm]{Proposition}
\newtheorem*{thm*}{Theorem}

\theoremstyle{definition}
\newtheorem{Rem}[thm]{Remark}
\newtheorem*{Rem*}{Remark}
\allowdisplaybreaks

\usepackage[dvipdfmx]{hyperref}
\hypersetup{
setpagesize=false,
 bookmarksnumbered=true,
 bookmarksopen=true,
 colorlinks=true,
 linkcolor=blue,
 citecolor=red,
}

\begin{document}
\begin{center}
\Large \textbf{Two-point local time penalizations with various clocks for \Levy\ 
processes}
\end{center}
\begin{center}
Kohki IBA (Graduate School of Science, Osaka University)\\
Kouji YANO (Graduate School of Science, Osaka University)
\end{center}
\begin{center}
\small \today
\end{center}

\begin{abstract}
Long-time limit of one-dimensional \Levy\ processes weighted and normalized with respect to the exponential functional of two-point local times are studied. The limit processes may vary according to the choice of random clocks.
\end{abstract}

\section{Introduction}
A penalization problem is to study the long-time limit of the form
\begin{align}
\label{1}
\lim_{\tau\to \infty}\frac{\mathbb{P}_x[F_s\cdot\Gamma_\tau]}{\mathbb{P}_x[\Gamma_\tau]},
\end{align}
where $(X=(X_t)_{t\ge 0},(\F_t)_{t\ge 0},(\mathbb{P}_x)_{x\in \R})$ is a Markov process, $(\Gamma_t)_{t\ge 0}$ is a non-negative process called a \emph{weight}, $(F_s)_{s\ge 0}$ is a process of test functions adapted to $(\F_s)_{s\ge 0}$, and $\tau$ is a net of parametrized random times tending to infinity, called a \emph{clock}.

To solve this problem, we want to find a $(\F_s)$-martingale $(M_s^\Gamma)_{s\ge 0}$ and a function $\rho(\tau)$ of the clock $\tau$ such that
\begin{align}
\label{2}
\lim_{\tau\to \infty}\rho(\tau)\mathbb{P}_x[F_s\cdot \Gamma_\tau]=\mathbb{P}_x[F_s\cdot M_s^\Gamma]
\end{align}
holds for $s\ge 0$, $x\in \R$, and bounded $\F_s$-measurable functions $F_s.$ If $M_0^\Gamma>0$ under $\P_x$, the convergence (\ref{2}) implies
\begin{align}
\label{2-1}
\lim_{\tau\to \infty}\frac{\P_x[F_s\cdot \Gamma_\tau]}{\P_x [\Gamma_\tau]}=\P_x\left[F_s\cdot \frac{M_s^\Gamma}{M_0^\Gamma}\right],
\end{align}
which solves the penalization problem (\ref{1}).

We consider a one-dimensional \Levy\ process $X$ which is recurrent and for which every point is regular for itself. Let $L_t^x$ denote the local time of $x$ up to time $t$ for $X$ (subject to a suitable normalization). The weight process we consider is given as
\begin{align}
\label{3}
\Gamma_{a,b,t}^{\lambda_a,\lambda_b}:=e^{-\lambda_aL_t^a-\lambda_bL_t^b}
\end{align}
for two distinct real points $a$ and $b$, and two positive constants $\lambda_a$ and $\lambda_b.$

For the characteristic exponent $\Psi(\lambda)$ of $X$, i.e., $\mathbb{P}_0[e^{i\lambda X_t}]=e^{-t\Psi(\lambda)}$, we always assume the condition
\begin{align*}
\textbf{(A)}\ \int_0^\infty \left|\frac{1}{q+\Psi(\lambda)}\right|d\lambda<\infty\qquad \mathrm{for\ each}\ q>0.
\end{align*}
Let $T_x:=\inf \{t>0;\ X_t=x\}$ denote the hitting time of a point $x\in \R$. For $-1\le \gamma\le 1$, we say\footnote{To describe the penalization limits, our limit $(c,d) \stackrel{(\gamma)}{\to}\infty$ is more suitable than the limit $(c,d)\stackrel{\gamma}{\to}\infty$ of the equation (1.9) of Takeda--Yano \cite{TY}.}
\begin{align}
(c,d)\stackrel{(\gamma)}{\to}\infty\ \mathrm{when}\ c\to \infty,\ d\to \infty,\ \mathrm{and}\ \frac{d-c}{c+d}\to \gamma.
\end{align}
Here for the random clock $\tau=(\tau_\lambda)$, we adopt one of the following:
\begin{enumerate}
\item exponential\ clock: $\tau=(\bm{e}_q)$ as $q\to 0+$, where $\bm{e}_q$ has the exponential distribution with its parameter $q>0$ and is independent of $X$;
\item hitting\ time\ clock: $\tau=(T_c)$ as $c\to \pm \infty$, where $T_c$ is the first hitting time at $c$;
\item two-point hitting time clock: $\tau=(T_c\wedge T_{-d})$ as $(c,d)\stackrel{(\gamma)}{\to}\infty$;
\item inverse\ local\ time\ clock: $\tau=(\eta_u^c)$ as $c\to \pm \infty$, where $\eta^c=(\eta_u^c)_{u\ge 0}$ is an inverse local time.
\end{enumerate}
Then, our main theorem is as follows (see Theorems \ref{c46}, \ref{d45}, \ref{e26}, \ref{f24}, and \ref{e56} for the details):
\begin{thm}
For distinct points $a,b\in \R$, for constants $\lambda_a,\lambda_b>0$, and for a constant $-1\le \gamma \le 1$, there exists a positive function $\varphi_{a,b}^{(\gamma),\lambda_a,\lambda_b}(x)$ such that the process
\begin{align}
\label{0-1}
\left(M_{a,b,s}^{(\gamma),\lambda_a,\lambda_b}:=\varphi_{a,b}^{(\gamma),\lambda_a,\lambda_b}(X_s)e^{-\lambda_aL_t^a-\lambda_b L_{t}^b}\right)_{s\ge 0}
\end{align}
is a martingale, and the following assertions hold:
\begin{enumerate}
\item exponential clock: $\displaystyle\lim_{q\to 0+}r_q(0)\mathbb{P}_x\left[F_s\cdot \Gamma_{a,b,\bm{e}_q}^{\lambda_a,\lambda_b}\right]= \mathbb{P}_x\left[F_s \cdot M_{a,b,s}^{(0),\lambda_a,\lambda_b}\right],$
\item hitting time clock: $\displaystyle \lim_{c\to \pm \infty }h^B(c)\mathbb{P}_x\left[F_s\cdot \Gamma_{a,b,T_c}^{\lambda_a,\lambda_b}\right]=\mathbb{P}_x\left[F_s \cdot M_{a,b,s}^{(\pm 1),\lambda_a,\lambda_b}\right],$
\item two-point hitting time clock: $\displaystyle \lim_{(c,d)\stackrel{(\gamma)}{\to}\infty}h^C(c,-d)\mathbb{P}_x\left[F_s\cdot \Gamma_{a,b,T_c\wedge T_{-d}}^{\lambda_a,\lambda_b}\right]= \mathbb{P}_x\left[F_s \cdot M_{a,b,s}^{(\gamma),\lambda_a,\lambda_b}\right],$
\item inverse local time clock: $\displaystyle  \lim_{c\to \pm \infty}h^B(c)\mathbb{P}_x\left[F_s\cdot \Gamma_{a,b,\eta_u^c}^{\lambda_a,\lambda_b}\right]=\mathbb{P}_x\left[F_s \cdot M_{a,b,s}^{(\pm 1),\lambda_a,\lambda_b}\right],$
\end{enumerate}
where $r_q$ is the $q$-resolvent density (see (\ref{b7})), $h^B$ is defined by (\ref{b34}), and $h^C$ is defined by (\ref{b46}).
\end{thm}

Let us explain the backgrounds of our results. Roynette--Vallois--Yor (\cite{RVY1} and \cite{RVY2}) have studied the penalization with \emph{constant clock} $\tau=t$ for a Brownian motion. Let $B=(B_t)_{t\ge 0}$ be a one-dimensional standard Brownian motion and $(\F_t)_{t\ge 0}$ denote the natural filtration of $B$. In particular, when 
\begin{align}
\label{a2}
\Gamma_s=f(L_s^0)
\end{align}
for a non-negative integrable function $f$ and with $L=(L_s^0)_{s\ge 0}$ being the local time at $0$ of $B$, this problem is called the \emph{local time penalization}. In this case, the martingale $M^\Gamma$ in (\ref{2}) is given by
\begin{align}
\label{a3}
M_s^\Gamma=f(L_s^0)|B_s|+\int_0^\infty f(L_s^0+u)du.
\end{align}
This result for a Brownian motion has been generalized to other processes. For example, we refer to Debs \cite{D} for simple random walks, Najnudel--Roynette--Yor \cite{NRY} for Bessel processes and Markov chains, and Yano--Yano--Yor \cite{YYY} for symmetric stable processes.

The local time penalization problem using random clocks date back to the \emph{taboo process} of Knight \cite{Kn}. It was the problem of conditioning to avoid zero, which is considered to be a special case of the local time penalization with weight $\Gamma_s=1_{\{L_s^0=0\}}.$ This problem has been generalized to one-dimensional \Levy\ processes by Chaumont--Doney \cite{CD} for conditioning to stay positive, by Pant\'{i} \cite{Pa} for conditioning to avoid zero, and to one-dimensional diffusions by Yano--Yano \cite{YY} for conditioning to avoid zero and Profeta--Yano--Yano \cite{PYY} for local time penalization.

Our penalizations with weight (\ref{3}) are an extension of the "one-point" local time penalization in \cite{TY} to the "two-point" one. The results of \cite{TY} were as follows:
\begin{thm}[Takeda--Yano \cite{TY}]
Suppose that the condition \textbf{(A)} is satisfied. For $x\in \R\ $and $-1\le \gamma \le 1,$ we define
\begin{align}
\label{d18}
h^{(\gamma)}(x):=h(x)+\frac{\gamma}{m^2}x,
\end{align}
where $m^2=\mathbb{P}_0[X_1^2]$ and the function $h$ is a renormalized zero resolvent (see Proposition \ref{b28}). Then, for a non-negative integrable function $f$, the process
\begin{align}
\label{d18-1}
\left(M_s^{(\gamma,f)}:=h^{(\gamma)}(X_s)f(L_s^0)+\int_0^\infty f(L_s^0+u)du\right)_{s\ge 0}
\end{align}
is a martingale. Moreover if $M_0^{(\gamma)}>0$ under $\mathbb{P}_x$, for any $\F_s$-measurable bounded function $F_s$, it holds the following\footnote{The results of Takeda--Yano \cite{TY} about the two-point hitting time clocks involve several minor computational errors, all of which we correct in this paper; see \cite{ISY} for the details.}:
\begin{enumerate}
\item exponential clock: $\displaystyle \lim_{q\to 0+}r_q(0)\mathbb{P}_x\left[F_s\cdot f(L_{\bm{e}_q}^0)\right]=\mathbb{P}_x\left[F_s\cdot M_s^{(0,f)}\right],$
\item hitting time clock: $\displaystyle \lim_{a\to \pm\infty}h^B(a)\mathbb{P}_x\left[F_s\cdot f(L_{T_a}^0)\right]=\mathbb{P}_x\left[F_s\cdot M_s^{(\pm 1,f)}\right],$
\item two-point hitting time clock: $\displaystyle \lim_{(a,b)\stackrel{(\gamma)}{\to}\infty}h^C(a,-b)\mathbb{P}_x\left[F_s\cdot f(L_{T_a\wedge {T_{-b}}}^0)\right]=\mathbb{P}_x\left[F_s\cdot M_s^{(\gamma,f)}\right],$
\item inverse local time clock: $\displaystyle \lim_{a\to \pm\infty}h^B(a)\mathbb{P}_x\left[F_s\cdot f(L_{\eta_u^a}^0)\right]=\mathbb{P}_x\left[F_s\cdot M_s^{(\pm 1,f)}\right].$
\end{enumerate}
\end{thm}
\begin{Rem}
Note that when $m^2=\infty,$  we have
\begin{align}
\label{d19}
h^{(\gamma)}(x)\equiv h^{(0)}(x)\equiv h(x)
\end{align}
for any $-1\le \gamma\le 1.$
\end{Rem}

\begin{Rem}
The function $\varphi_{a,b}^{(\gamma),\lambda_a,\lambda_b}$ is given explicitly as follows (see Proposition \ref{e8}):
\begin{align*}
\label{c41}
\varphi_{a,b}^{(\gamma),\lambda_a,\lambda_b}(x)&=h^{(\gamma)}(x-a)-\mathbb{P}_x(T_b<T_a)h^{(\gamma)}(b-a)\\
&\qquad +\mathbb{P}_x(T_a<T_b)\cdot \frac{h^{(\gamma)}(a-b)}{1+\lambda_ah^B(a-b)}\\
&\qquad +\mathbb{P}_x(T_a<T_b)\cdot\frac{1}{1+\lambda_ah^B(b-a)}\cdot \frac{1+\lambda_a h^{(\gamma)}(b-a)}{\lambda_a+\lambda_b+\lambda_a\lambda_b h^B(a-b)}\\
&\qquad +\mathbb{P}_x(T_b<T_a)\cdot \frac{h^{(\gamma)}(b-a)}{1+\lambda_bh^B(a-b)}\\
&\qquad +\mathbb{P}_x(T_b<T_a)\cdot\frac{1}{1+\lambda_bh^B(b-a)}\cdot \frac{1+\lambda_b h^{(\gamma)}(a-b)}{\lambda_a+\lambda_b+\lambda_a\lambda_b h^B(a-b)}.
\stepcounter{equation}\tag{\theequation}
\end{align*}
Note that $\varphi_{a,b}^{(\gamma),\lambda_a,\lambda_b}(x)$ is symmetric with respect to $a$ and $b$, i.e., for $x\in \R,$
\begin{align}
\varphi_{a,b}^{(\gamma),\lambda_a,\lambda_b}(x)=\varphi_{b,a}^{(\gamma),\lambda_a,\lambda_b}(x).
\end{align}
When $m^2=\infty$, by (\ref{d19}), we have $\varphi_{a,b}^{(\gamma),\lambda_a,\lambda_b}(x)\equiv \varphi_{a,b}^{(0),\lambda_a,\lambda_b}(x)$ and hence $M_{a,b,t}^{(\gamma),\lambda_a,\lambda_b}=M_{a,b,t}^{(0),\lambda_a,\lambda_b}$ for any $-1\le \gamma\le 1.$
\end{Rem}

\begin{Rem}
\label{c45-1}
By letting $a=0,\ \lambda_a=\lambda>0,$ and letting $\lambda_b\to 0+$ in the definition (\ref{0-1}), we have
\begin{align*}
\label{c45-2}
M_{0,b,t}^{(\gamma),\lambda,0}:=\lim_{\lambda_b\to 0+}M_{0,b,t}^{(\gamma),\lambda,\lambda_b}=\left(h^{(\gamma)}(X_t)+\frac{1}{\lambda}\right)e^{-\lambda L_t}.
\stepcounter{equation}\tag{\theequation}
\end{align*}
This coincides with $M_t^{(\gamma,f)}$ in (\ref{d18-1}) with $f(x)=e^{-\lambda x}.$
\end{Rem}

In \cite{RVY2}, they consider the penalization problem when the weight process is given as
\begin{align}
\label{a7}
\Gamma_t=\exp \left\{-\int_{\R}L_t^xq(dx)\right\}
\end{align}
for a positive Radon measure $q(dx)$ on $\R$, which satisfies $0<\int_{\R} (1+|x|)q(dx)<\infty.$ This problem is called the \emph{Kac killing penalization}. Our weight process (\ref{3}) can be considered to be the case
\begin{align}
\label{a8}
q=\lambda_a\delta_a+\lambda_b\delta_b.
\end{align}
The case $\lambda_a>0$ and $\lambda_b=\infty$ is formally considered a one-point local time penalization with conditioning to avoid another point. The case $\lambda_a=\lambda_b=\infty$ is formally considered a conditioning to avoid two points. These cases are discussed in \cite{ISY}.

\subsection*{Organization}
This paper is organized as follows. In Section \ref{S2}, we prepare some general results of \Levy\ processes. In Sections \ref{S3}, \ref{S4}, \ref{S5}, and \ref{S6}, we discuss the penalization results with exponential clock, hitting time clock, two-point hitting time clock, and inverse local time clock, respectively.


\section{Preliminaries}
\label{S2}
\subsection{\Levy\ process and resolvent density}
Let $(X,\mathbb{P}_x)$ be the canonical representation of a real valued \Levy\ process starting from $x\in \R$ on the \cadlag\ path space. For $t>0$, we denote by $\F_t^X=\sigma(X_s,\ 0\le s\le t)$ the natural filtration of $X$ and write $\F_t=\bigcap_{s>t}\F_s^X$. For $a\in \R$, let $T_a$ be the hitting time of $\{a\}$ for $X$, i.e.,
\begin{align}
\label{b1}
T_a=\inf\{t>0:\ X_t=a\}.
\end{align}
For $\lambda\in \R$, we denote by $\Psi(\lambda)$ the characteristic exponent of $X$, which satisfies 
\begin{align}
\label{b2}
\mathbb{P}_0\left[e^{i\lambda X_t}\right]=e^{-t\Psi(\lambda)}
\end{align}
for $t\ge 0.$ Moreover, by \Levy--Khinchin formula, it is denoted by
\begin{align}
\label{b3}
\Psi(\lambda)=iv\lambda+\frac{1}{2}\sigma^2 \lambda^2+\int_{\R}\left(1-e^{i\lambda x}+i \lambda x 1_{\{|x|<1\}}\right)\nu(dx),
\end{align}
where $v\in \R,\ \sigma\ge 0,\ $and $\nu$ is a measure on $\R$, called a \emph{\Levy\ measure}, with $\nu(\{0\})=0$ and $\int_{\R}(x^2\wedge 1)\nu(dx)<\infty.$

Throughout this paper, we always assume $(X,\mathbb{P}_0)$ is recurrent, i.e., 
\begin{align}
\label{b4}
\mathbb{P}_0\left[\int_0^\infty 1_{\{|X_t-a|<\e\}}dt\right]=\infty
\end{align}
for all $a\in \R$ and $\e>0$, and always assume the condition
\begin{align*}
\textbf{(A)}\ \int_0^\infty \left|\frac{1}{q+\Psi(\lambda)}\right|d\lambda<\infty\qquad \mathrm{for\ each}\ q>0.
\end{align*}
Then, the following conditions hold (Bretagnolle \cite{L1-1} and Kesten \cite{L1-2}):
\begin{enumerate}
\item[(1)] The process $X$ is not a compound Poisson process;
\item[(2)] $0$ is regular for itself, i.e., $\mathbb{P}_0(T_0=0)=1;$
\item[(3)] We have either $\sigma>0$ or $\displaystyle \int_{(-1,1)}|x|\nu(dx)=\infty$.
\end{enumerate}
It is known that $X$ has a bounded continuous resolvent density $r_q$:
\begin{align}
\label{b7}
\int_{\R}f(x)r_q(x)dx=\mathbb{P}_0\left[\int_0^\infty e^{-qt}f(X_t)dt\right]
\end{align}
holds for $q>0$ and non-negative measurable functions $f$. See, e.g., Theorems II.16 and II.19 of \cite{Ber}. Moreover, there exists an equality that connects the hitting time of $0$ and the resolvent density:
\begin{align}
\label{b9}
\mathbb{P}_x\left[e^{-qT_0}\right]=\frac{r_q(-x)}{r_q(0)}
\end{align}
for $q>0$ and $x\in \R.$ See, e.g., Corollary II.18 of \cite{Ber}. If condition \textbf{(A)} holds, then Tukada (Corollary 15.1 of \cite{Tukada}) showed that the resolvent density can be expressed as
\begin{align}
\label{b11}
r_q(x)=\frac{1}{2\pi}\int_{-\infty}^\infty \frac{e^{-i\lambda x}}{q+\Psi(\lambda)}d\lambda =\frac{1}{\pi}\int_0^\infty \Re \left(\frac{e^{-i\lambda x}}{q+\Psi(\lambda)}\right)d\lambda
\end{align}
for all $q>0$ and $x\in \R.$


\subsection{Local time and excursion}
We denote by $\mathcal{D}$ the set of \cadlag\ paths $e:[0,\infty)\to \R\cup\{\Delta\}$ such that
\begin{align}
\label{b12}
\begin{cases}
e(t)\in \R\setminus \{0\}&\mathrm{for}\ 0<t<\zeta(e),\\
e(t)=\Delta &\mathrm{for}\ t\ge \zeta(e),
\end{cases}
\end{align}
where the point $\Delta$ is an isolated point and $\zeta$ is the excursion length, i.e.,
\begin{align}
\label{b13}
\zeta=\zeta(e):=\inf \{t>0:\ e(t)=\Delta\}.
\end{align}
Let $\Sigma$ denote the $\sigma$-algebra on $\mathcal{D}$ generated by cylinder sets.

Assume the condition \textbf{(A)} holds. Then, we can define a local time at $a\in \R$, which we denote by $L^a=(L_t^a)_{t\ge 0}$. It is defined by
\begin{align}
\label{b14}
L_t^a:=\lim_{\e\to 0+}\frac{1}{2\e}\int_0^t1_{\{|X_s-a|<\e\}}ds.
\end{align}
It is known that $L^a$ is continuous in $t$ and satisfies 
\begin{align}
\label{b15}
\mathbb{P}_x\left[\int_0^\infty e^{-qt}dL_t^a\right]=r_q(a-x)
\end{align}
for $q>0$ and $x\in \R$. See, e.g., Section V of \cite{Ber}. In particular, from this expression, $r_q(x)$ is non-decreasing as $q\to 0+.$

Let $\eta^a=(\eta_l^a)_{l\ge 0}$ be an inverse local time, i.e.,
\begin{align}
\label{b16}
\eta_l^a:=\inf\{t>0:\ L_t^a>l\}.
\end{align}
It is known that the process $(\eta^a,\mathbb{P}_a)$ is a possibly killed subordinator which has the Laplace exponent
\begin{align}
\label{b17}
\mathbb{P}_a\left[e^{-q\eta_l^a}\right]=e^{-\frac{l}{r_q(0)}}
\end{align}
for $l>0$ and $q>0$. See, e.g., Proposition V.4 of \cite{Ber}.

We denote $\epsilon_l^a$ for an excursion away from $a\in \R$ which starts at local time $l\ge 0$, i.e.,
\begin{align}
\label{b18}
\epsilon_l^a(t):=\begin{cases}
X_{t+\eta_{l-}^a}&\mathrm{for}\ 0\le t<\eta_{l}^a-\eta_{l-}^a,\\
\Delta&\mathrm{for}\ t\ge \eta_l^a-\eta_{l-}^a.
\end{cases}
\end{align}
Then, $(\epsilon_l^a)_{l\ge 0}$ is a Poisson point process, and we write $n^a$ for the characteristic measure of $\epsilon^a$. It is known that $(\mathcal{D},\Sigma,n^a)$ is a $\sigma$-finite measure space. See, e.g., Section IV of \cite{Ber}. For $B\in \mathscr{B}(0,\infty)\otimes \Sigma$, we define
\begin{align}
\label{b19}
N^a(B):=\#\{(l,e)\in B:\ \epsilon_l^a=e\}.
\end{align}
Then, $N^a$ is a Poisson random measure with its intensity measure $ds\times n^a(de).$ It is known that the subordinator $\eta^0$ has no drift and its \Levy\ measure is $n^0(T_0\in dx)$. Thus, we have
\begin{align}
\label{b20}
e^{-\frac{l}{r_q(0)}}=\mathbb{P}_0\left[e^{-q\eta_l^0}\right]=e^{-ln^0[1-e^{-qT_0}]}
\end{align}
for $l\ge 0.$ This implies that
\begin{align}
\label{b21}
n^0\left[1-e^{-qT_0}\right]=\frac{1}{r_q(0)}.
\end{align}
Now, we set 
\begin{align}
\label{b22}
\kappa:=\lim_{q\to 0+}\frac{1}{r_q(0)}=n^0(T_0=\infty).
\end{align}
It is known that $\kappa=0$ if and only if $X$ is recurrent. See, e.g., Theorem I.17 of \cite{Ber} and Theorem 37.5 of \cite{Sato}.

For the excursion measure, the following famous equality, called the compensation formula, is known: if a function $F:(0,\infty)\times \mathcal{D}\times \mathcal{D}\to [0,\infty)$ satisfies the following: 
\begin{enumerate}
\item $F$ is a measurable function;
\item For each $t>0$, $(\omega,e)\mapsto F(t,\omega,e)$ is $\F_t\otimes \Sigma$-measurable;
\item For each $e\in \mathcal{D}$, $(F(t,\cdot,e),\ t\ge 0)$ is almost surely a left-continuous process.
\end{enumerate}
Then, for $t\ge 0$, it holds that
\begin{align}
\label{b24}
\mathbb{P}_0\left[\int_{(0,t]}\int_{\mathcal{D}}F(s,X,e)N^0(ds\otimes de)\right]=\mathbb{P}_0\left[\int_0^t \int_{\mathcal{D}}F(s,X,e)n^0(de)ds\right].
\end{align}
See, e.g., Theorem 4.4 of \cite{Kyp}. Moreover, the excursion measure $n^0$ has the following form of the Markov property: it holds that for any stopping time $T<\infty$, any non-negative $\F_t$-measurable functional $Z_t$, and any non-negative measurable functional $F$ on $\mathcal{D}$,
\begin{align}
\label{b26}
n^0\left[Z_t F(X\circ \theta_T)\right]=\int n^0[Z_t;\ X_T\in dx]\mathbb{P}_x^0[F(X)],
\end{align}
where $\mathbb{P}_x^0$ is the distribution of the killed process upon $T_0.$ See, e.g., Theorem III.3.28 of \cite{Blu}.


\subsection{The renormalized zero resolvent}
We define
\begin{align}
\label{b27}
h_q(x):=r_q(0)-r_q(-x)
\end{align}
for $q>0$ and $x\in \R.$ By (\ref{b11}), we have
\begin{align}
h_q(x)=\frac{1}{\pi}\int_0^\infty \Re \left(\frac{1-e^{i\lambda x}}{q+\Psi(\lambda)}\right)d\lambda.
\end{align}
It is clear that $h_q(0)=0$, and by (\ref{b9}), we have $h_q(x)\ge 0$. The following theorem plays a key role in our penalization results. Recall that $X$ is assumed recurrent.
\begin{prop}[Theorem 1.1 of \cite{TY}]
\label{b28}
If the condition $\textbf{(A)}$ is satisfied, then the following assertions hold:
\begin{enumerate}[(i)]
 \item For any $x\in \R$, $\displaystyle h(x):=\lim_{q\to 0+}h_q(x)$ exists and is finite;
 \item $h$ is continuous;
 \item $h$ is subadditive, that is, for $x,y\in \R,$ it holds that
\begin{align}
\label{b28-1}
h(x+y)\le h(x)+h(y).
\end{align}
\end{enumerate}
\end{prop}
We call $h$ the \emph{renormalized zero resolvent}.
\begin{prop}[Theorem 1.2 of \cite{TY}]
\label{b29}
If the condition $\textbf{(A)}$ is satisfied, then the following assertions hold:
\begin{enumerate}[(i)]
 \item It holds that
 \begin{align}
 \label{b29-0}
\displaystyle \lim_{x\to \pm \infty}\frac{h(x)}{|x|}=\frac{1}{m^2}\in [0,\infty);
\end{align}
 \item For all $x\in \R$, it holds that
 \begin{align}
 \label{b29-1}
\lim_{y\to \pm \infty}\{h(x+y)-h(y)\}=\pm \frac{x}{m^2},
\end{align}
\end{enumerate}
where $m^2=P[X_1^2]\in (0,\infty]$.
\end{prop}

\begin{prop}[Theorem 15.2 of \cite{Tukada}]
\label{b30}
For $t\ge 0$, it hold that
\begin{align}
\label{b31}
h_q(X_t)\to h(X_t)\ \mathrm{in}\ L^1(\mathbb{P}_x)
\end{align}
as $q\to 0+.$
\end{prop}


\subsection{The function $h^B$}
We will introduce the functions $h_q^B$ and $h^B.$
\begin{prop}[Lemma 3.5 of \cite{TY}]
\label{b32}
The following assertions hold:
\begin{enumerate}[(i)]
\item For $a\in \R$, it holds that
\begin{align}
\label{b33}
h_q^B(a):=\mathbb{P}_0\left[\int_0^{T_a}e^{-qt}dL_t^0\right]=h_q(a)+h_q(-a)-\frac{h_q(a)h_q(-a)}{r_q(0)}.
\end{align}
Consequently, it holds that
\begin{align}
\label{b34}
h^B(a):=\lim_{q\to 0+}h_q^B(a)=\mathbb{P}_0[L_{T_a}]=h(a)+h(-a);
\end{align}
\item For $x\in \R$, $q>0$, and distinct points $a,b\in \R,$ it holds that
\begin{align}
\label{b35}
\mathbb{P}_x\left[e^{-qT_a};\ T_a<T_b\right]=\frac{h_q(b-a)+h_q(x-b)-h_q(x-a)-\frac{h_q(x-b)h_q(b-a)}{r_q(0)}}{h_q^B(a-b)}.
\end{align}
Consequently, it holds that
\begin{align}
\label{b36}
\mathbb{P}_x(T_a<T_b)=\frac{h(b-a)+h(x-b)-h(x-a)}{h^B(a-b)}.
\end{align}
\end{enumerate}
\end{prop}

\begin{prop}[Lemma 3.7 of \cite{TY}]
\label{b37}
It holds that 
\begin{align}
\label{b38}
\lim_{x\to \infty}h^B(x)=\infty.
\end{align}
\end{prop}

\begin{prop}[Lemma 3.8 of \cite{TY}]
\label{b39}
For $a\in \R\setminus \{0\}$, it holds that 
\begin{align}
\label{b40}
 n^0(T_a<T_0)=\frac{1}{h^B(a)}.
\end{align}
\end{prop}


\subsection{Various expectations of local time at random times}
In this subsection, we present the results of various expectations of local time clock that we will use frequently later.
\begin{prop}[Lemma 4.1 of \cite{TY}]
\label{b41}
Let $f$ be a non-negative measurable function. Then, it holds that for $q>0$ and $x\in \R$,
\begin{align}
\label{b42}
\mathbb{P}_x[f(L_{\bm{e}_q}^0)]=\frac{1}{r_q(0)}\left\{h_q(x)f(0)+\left(1-\frac{h_q(x)}{r_q(0)}\right)\int_0^\infty e^{-\frac{u}{r_q(0)}}f(u)du\right\}.
\end{align}
\end{prop}

\begin{prop}[Lemma 5.1 of \cite{TY}]
\label{b43}
For $x\in \R,\ a\in \R\setminus \{0\}$, and non-negative measurable function $f$, we have
\begin{align}
\label{b44}
\mathbb{P}_x\left[f(L_{T_a}^0)\right]=\mathbb{P}_x(T_0>T_a)f(0)+\frac{\mathbb{P}_x(T_0<T_a)}{h^B(a)}\int_0^\infty e^{-\frac{u}{h^B(a)}}f(u)du.
\end{align}
\end{prop}

\begin{prop}[Lemma 6.1 of \cite{TY}]
\label{b45}
For distinct points $a,b\in \R$, it holds that
\begin{align*}
\label{b46}
h^C(a,b):&=\mathbb{P}_0[L_{T_a\wedge T_b}^0]\\
&=\frac{1}{h^B(a-b)}\left\{\begin{aligned}
&(h(b)+h(-a))h(a-b)+(h(a)+h(-b))h(b-a)\\
&\qquad +(h(a)-h(b))(h(-b)-h(-a))-h(a-b)h(b-a)
\end{aligned}\right\}.
\stepcounter{equation}\tag{\theequation}
\end{align*}
\end{prop}

\begin{prop}[Lemma 6.2 of \cite{TY}]
\label{b47}
For $x\in \R$ and distinct points $a,b,c\in \R$, it hold that
\begin{align*}
\label{b48}
&\mathbb{P}_x\left[e^{-qT_a};\ T_a<T_b\wedge T_c\right]\\
&=\frac{\mathbb{P}_x\left[e^{-qT_a};\ T_a<T_b\right]-\mathbb{P}_x\left[e^{-qT_c};\ T_c<T_b\right]-\mathbb{P}_c\left[e^{-qT_a};\ T_a<T_b\right]}{1-\mathbb{P}_a\left[e^{-qT_c};\ T_c<T_b\right]\mathbb{P}_c\left[e^{-qT_a};\ T_a<T_b\right]}.
\stepcounter{equation}\tag{\theequation}
\end{align*}
Consequently, it holds that
\begin{align}
\label{b49}
\mathbb{P}_x(T_a<T_b\wedge T_c)=\frac{\mathbb{P}_x(T_a<T_b)-\mathbb{P}_x(T_c<T_b)\mathbb{P}_c(T_a<T_b)}{1-\mathbb{P}_a(T_c<T_b)\mathbb{P}_c(T_a<T_b)}.
\end{align}
\end{prop}


\section{Exponential clock}
\label{S3}
From this section to the end, we write simply
\begin{align}
\Gamma_t:=e^{-\lambda_aL_t^a-\lambda_bL_t^b}.
\end{align}
Let us find the limit of $r_q(0)\mathbb{P}_{x}\left[\Gamma_{\bm{e}_q}\right]$ as $q\to 0+.$

\begin{prop}
\label{c40}
For distinct points $a,b\in \R$ and for constants $\lambda_a,\lambda_b>0$ it holds that
\begin{align}
\label{c42}
\lim_{q\to 0+}r_q(0)\mathbb{P}_x\left[\Gamma_{\bm{e}_q}\right]=\varphi_{a,b}^{(0),\lambda_a,\lambda_b}(x)
\end{align}
for all $x\in \R$, and 
\begin{align}
\label{c42-1}
\lim_{q\to 0+}r_q(0)\mathbb{P}_{X_t}\left[\Gamma_{\bm{e}_q}\right]=\varphi_{a,b}^{(0),\lambda_a,\lambda_b}(X_t)\ \mathrm{a.s.\ and\ in\ } L^1(\mathbb{P}_x).
\end{align}
\end{prop}

Before proving Proposition \ref{c40}, we prove the following lemma.

\begin{lem}
\label{c5}
For distinct points $a,b\in \R$, for constants $\lambda_a,\lambda_b>0$, and for $q>0$, it holds that
\begin{align}
\label{c6}
I_{a,b}^q:=\mathbb{P}_a\left[\int_{0}^{T_b} e^{-\lambda_aL_{s}^a}qe^{-qs}ds\right]=\frac{1-n^a[e^{-qT_b}]h_q(b-a)}{1+\lambda_a r_q(0)+n^a[e^{-qT_b}] r_q(a-b)}.
\end{align}
\end{lem}
\begin{proof}
By dividing the range of integration into excursion intervals, we have
\begin{align}
\label{c7}
I_{a,b}^q=\mathbb{P}_a\left[\sum_{l\le \sigma_{\{T_b<\infty\}}^a}\int_{\eta_{l-}^a}^{\eta_l^a\wedge T_b}e^{-\lambda_a l}qe^{-qs}ds\right],
\end{align}
where for $A\in \Sigma $, $\sigma_A^a$ is the first hitting time for the Poisson point process $(\epsilon_t^a)_{t\ge 0}:$
\begin{align}
\sigma_{A}^a:=\inf\{t\ge 0:\ \epsilon_t^a\in A\}.
\end{align}
Then, we have
\begin{align*}
\label{c10}
\displaystyle &\mathbb{P}_a\left[\int_{s\le \sigma_A^a} \int_{\mathcal{D}}F(\eta_{s-}^a,X,e)N^a(ds\otimes de)\right]\\
&=\mathbb{P}_a\left[\left(\int_{s< \sigma_A^a}+\int_{s=\sigma_A^a}\right) \int_{\mathcal{D}}F(\eta_{s-}^a,X,e)N^a(ds\otimes de)\right]=:(a)+(b).
\stepcounter{equation}\tag{\theequation} 
\end{align*}
Now, we let 
\begin{align}
\label{c11}
N^{a,A}(\cdot):=N^a(\cdot \cap \{(0,\infty)\times A\})
\end{align}
and let 
\begin{align}
\label{c12}
\eta_s^{a,A}:=\displaystyle \int_{(0,s]}\int_\mathcal{D}\zeta (e)N^{a,A}(ds\otimes de),
\end{align}
where $\zeta(e)$ is the excursion length of $e$. Since $\sigma_A^a$ and $N^{a,A^c}$ are independent, we have by the compensation formula,
\begin{align*}
\label{c13}
(a)&=\displaystyle \mathbb{P}_a\left[\int_{s<\sigma_A^a}\int_{\mathcal{D}}F(\eta_{s-}^{a,A^c},X,e)N^{a,A^c}(ds\otimes de)\right]\\
&=\displaystyle \mathbb{P}_a\left[\mathbb{P}_a\left[\int_{s\le t}\int_{\mathcal{D}}F(\eta_{s-}^{a,A^c},X,e)N^{a,A^c}(ds\otimes de)\right]\Bigg|_{t=\sigma_A^a}\right]\\
&=\mathbb{P}_a\left[\int_0^{\sigma_A^a}\int_{\mathcal{D}}F(\eta_{s-}^{a,A^c},X,e)n^a(de\cap A^c)ds\right].
\stepcounter{equation}\tag{\theequation} 
\end{align*}
Moreover, since $\sigma_A^a$ has the exponential distribution with its parameter $n^a(A)$, and since $\epsilon_{\sigma_A^a}^a$ has the distribution $n^a(\cdot |A)$, we have
\begin{align*}
\label{c14}
(b)&=\displaystyle \mathbb{P}_a\left[\int_{s=\sigma_A^a}\int_{\mathscr{D}}F(\eta_{s-}^{a,A^c},X,e)N^{a,A}(ds\otimes de)\right]\\
&=\mathbb{P}_a\left[F(\eta_{\sigma_A^a-}^{a,A^c},X,\epsilon_{\sigma_A^a}^a)\right]\\
&=\mathbb{P}_a\left[\int_0^\infty ds\int_{\mathcal{D}}F(\eta_{s-}^{a,A^c},X,e)n^a(A)e^{-n^a(A)s}n^a(de|A)\right]\\
&=\mathbb{P}_a\left[\int_0^\infty ds\int_{\mathcal{D}}F(\eta_{s-}^{a,A^c},X,e)e^{-n^a(A)s}n^a(de\cap A)\right].
\stepcounter{equation}\tag{\theequation} 
\end{align*}
Thus, by (\ref{c7}), (\ref{c10}), (\ref{c13}), and (\ref{c14}), we have
\begin{align*}
\label{c15}
I_{a,b}^q&=\mathbb{P}_a\left[\sum_{l\le \sigma_{\{T_b<\infty\}}^a}\int_0^{\eta_l^a\wedge T_b-\eta_{l-}^a}e^{-\lambda_al}qe^{-q(s+\eta_{l-}^a)}ds\right]\\
&=\mathbb{P}_a\left[\int_{l\le \sigma_{\{T_b<\infty\}}^a}\int_{\mathcal{D}}\left(\int_0^{\eta_l^a\wedge T_b-\eta_{l-}^a}e^{-\lambda_al}qe^{-q(s+\eta_{l-}^a)}ds\right)N^a(dl\otimes de)\right]\\
&=\mathbb{P}_a\left[\int_0^{\sigma_{\{T_b<\infty\}}^a}\int_{\mathcal{D}}\left(\int_0^{T_a}e^{-\lambda_al}qe^{-q(s+\eta_{l-}^{a,\{T_b=\infty\}})}ds\right)n^a(de\cap \{T_b=\infty\})dl\right]\\
&\qquad +\mathbb{P}_a\left[\int_0^\infty dl\int_{\mathcal{D}}\left(\int_0^{T_b}e^{-\lambda_al}qe^{-q(s+\eta_{l-}^{a,\{T_b=\infty\}})}ds\right)e^{-n^a(T_b<\infty)l}n^a(de\cap \{T_b<\infty\})\right]\\
&=\mathbb{P}_a\left[\int_0^{\sigma_{\{T_b<\infty\}}^a}e^{-\lambda_al}e^{-q\eta_{l-}^{a,\{T_b=\infty\}}}dl\right]n^a(1-e^{-qT_a},\ T_b=\infty)\\
&\qquad +\mathbb{P}_a\left[\int_0^\infty e^{-\lambda_a l}e^{-q\eta_{l-}^{a,\{T_b=\infty\}}}e^{-n^a(T_b<\infty)l}dl\right]n^a(1-e^{-qT_b}, T_b<\infty)\\
&=\left(\int_0^\infty dt\int_0^te^{-\lambda_al}\mathbb{P}_a\left[e^{-q\eta_{l-}^{a,\{T_b=\infty\}}}\right]n^a(T_b<\infty)e^{-n^a(T_b<\infty)t}dl\right)n^a(1-e^{-qT_a},\ T_b=\infty)\\
&\qquad +\left(\int_0^\infty e^{-\lambda_a l}\mathbb{P}_a\left[e^{-q\eta_{l-}^{a,\{T_b=\infty\}}}\right]e^{-n^a(T_b<\infty)l}dl\right)n^a(1-e^{-qT_b}, T_b<\infty).
\stepcounter{equation}\tag{\theequation} 
\end{align*}
Since $(\eta_l^a)_{l\ge 0}$ is a subordinator with its \Levy\ measure $n^a(T_a\in dx)$ and no drift, by \Levy--Khinchin formula, we have
\begin{align*}
\label{c16}
\displaystyle \mathbb{P}_a\left[e^{-q\eta_{l-}^{a,\{T_b=\infty\}}}\right]&=\mathbb{P}_a\left[e^{-q\eta_{l}^{a,\{T_b=\infty\}}}\right]\\
&=\exp \left\{-l\int_{(0,\infty)} (1-e^{-qx})n^a(T_a\in dx \cap \{T_b=\infty\})\right\}\\
&=\exp \left\{-ln^a[1-e^{-qT_a};\ T_b=\infty]\right\}.
\stepcounter{equation}\tag{\theequation} 
\end{align*}
By the strong Markov property of the excursion measure $n^a$ and by (\ref{b9}), we have
\begin{align*}
\label{c17}
\displaystyle n^a[e^{-qT_a};\ T_b<\infty]&=\displaystyle \int n^a\left[e^{-qT_b}1_{\{T_b<T_a\}};\ X_{T_b}\in dx\right]\mathbb{P}_x^{a}\left[e^{-qT_a}\right]\\
&=n^a[e^{-qT_b};\ T_b<\infty]\mathbb{P}_b\left[e^{-qT_a}\right]\\
&=n^a[e^{-qT_b}]\cdot\frac{r_q(a-b)}{r_q(0)},
\stepcounter{equation}\tag{\theequation} 
\end{align*}
where $\mathbb{P}_x^a$ is the distribution of the killed process upon $T_a$. Thus, by (\ref{b21}), (\ref{b40}), and  (\ref{c17}), we have
\begin{align*}
\label{c18}
\displaystyle n^a[1-e^{-qT_a};\ T_b=\infty]&=n^a[1-e^{-qT_a}]-n^a[T_b<\infty]+n^a[e^{-qT_a};\ T_b<\infty]\\
&=\frac{1}{r_q(0)}-\frac{1}{h^B(b-a)}+n^a[e^{-qT_b}]\cdot \frac{r_q(a-b)}{r_q(0)}.
\stepcounter{equation}\tag{\theequation} 
\end{align*}
Consequently, by (\ref{c15}), (\ref{c16}), and (\ref{c18}), we obtain
\begin{align*}
\label{c19}
I_{a,b}^q&=\left(\int_0^\infty \int_0^te^{-\lambda_al}e^{-ln^a[1-e^{-qT_a};\ T_b=\infty]}dl\cdot \frac{1}{h^B(b-a)}\cdot e^{-\frac{t}{h^B(b-a)}}dt\right)n^a[1-e^{-qT_a};\ T_b=\infty]\\
&\qquad +\left(\int_0^\infty e^{-\lambda_a l}e^{-ln^a[1-e^{-qT_a};\ T_b=\infty]}e^{-\frac{l}{h^B(b-a)}}dl\right)n^a[1-e^{-qT_b};\ T_b<\infty]\\
&=\frac{1-n^a[e^{-qT_b}]h_q(b-a)}{1+\lambda_a r_q(0)+n^a[e^{-qT_b}] r_q(a-b)}.
\stepcounter{equation}\tag{\theequation} 
\end{align*}
The proof is complete.
\end{proof}

Let us proceed to the proof of Proposition \ref{c40}.

\begin{proof}[Proof of Proposition \ref{c40}]
First, by dividing the range of integration, we have
\begin{align*}
\label{c1}
\mathbb{P}_x\left[\Gamma_{\bm{e}_q}\right]&=\mathbb{P}_x\left[\int_0^\infty e^{-\lambda_aL_s^a-\lambda_b L_s^b}qe^{-qs}ds\right]\\
&=\mathbb{P}_x\left[\int_0^{T_a\wedge T_b}qe^{-qs}ds\right]\\
&\qquad +\mathbb{P}_x\left[\int_{T_a}^\infty e^{-\lambda_aL_s^a-\lambda_b L_s^b}qe^{-qs}ds,\ T_a<T_b\right]\\
&\qquad +\mathbb{P}_x\left[\int_{T_b}^\infty e^{-\lambda_aL_s^a-\lambda_b L_s^b}qe^{-qs}ds,\ T_b<T_a\right].
\stepcounter{equation}\tag{\theequation}
\end{align*}
In the first term of (\ref{c1}), by a simple calculation, we have
\begin{align}
\label{c3}
\mathbb{P}_x\left[\int_0^{T_a\wedge T_b}qe^{-qs}ds\right]=1-\mathbb{P}_x\left[e^{-q(T_a\wedge T_b)}\right].
\end{align}

Next, we consider the second term of (\ref{c1}). By the strong Markov property, we have
\begin{align*}
\label{c4}
&\mathbb{P}_x\left[\int_{T_a}^\infty e^{-\lambda_aL_s^a-\lambda_b L_s^b}qe^{-qs}ds,\ T_a<T_b\right]\\
&=\mathbb{P}_x\left[\int_{T_a}^{T_b} e^{-\lambda_aL_s^a}qe^{-qs}ds,\ T_a<T_b\right]+\mathbb{P}_x\left[\int_{T_b}^\infty e^{-\lambda_aL_s^a-\lambda_b L_s^b}qe^{-qs}ds,\ T_a<T_b\right]\\
&=\mathbb{P}_x\left[\int_{0}^{T_b-T_a} e^{-\lambda_a(L_{s+T_a}^a-L_{T_a}^a)}qe^{-q(s+T_a)}ds,\ T_a<T_b\right]\\
&\qquad +\mathbb{P}_x\left[\int_{0}^\infty e^{-\lambda_a(L_{s+T_b}^a-L_{T_b}^a+L_{T_b}^a)-\lambda_b (L_{s+T_b}^b-L_{T_b}^b)}qe^{-q(s+T_b)}ds,\ T_a<T_b\right]\\
&=\mathbb{P}_x\left[e^{-qT_a},\ T_a<T_b\right]\mathbb{P}_a\left[\int_{0}^{T_b} e^{-\lambda_aL_{s}^a}qe^{-qs}ds\right]+\mathbb{P}_x\left[e^{-\lambda_a L_{T_b}^a}e^{-qT_b},\ T_a<T_b\right]\mathbb{P}_b\left[\Gamma_{\bm{e}_q}\right].
\stepcounter{equation}\tag{\theequation}
\end{align*}

Here, we have
\begin{align}
\label{c21}
\mathbb{P}_a\left[\Gamma_{\bm{e}_q}\right]=\frac{I_{a,b}^q+J_{a,b}^qI_{b,a}^q}{1-J_{a,b}^qJ_{b,a}^q},
\end{align}
where
\begin{align}
\label{c22}
J_{a,b}^q:&=\displaystyle \mathbb{P}_a\left[e^{-qT_b}e^{-\lambda_aL_{T_b}^a}\right].
\end{align}
Indeed, by the strong Markov property at $T_b$, we have
\begin{align*}
\label{c23}
\mathbb{P}_a\left[\Gamma_{\bm{e}_q}\right]&=\mathbb{P}_a\left[\int_0^{T_b}e^{-\lambda_aL_s^a}qe^{-qs}ds\right]+\mathbb{P}_a\left[\int_{T_b}^\infty e^{-\lambda_aL_s^a-\lambda_b L_s^b}qe^{-qs}ds\right]=I_{a,b}^{q}+J_{a,b}^q\mathbb{P}_b\left[\Gamma_{\bm{e}_q}\right].
\stepcounter{equation}\tag{\theequation} 
\end{align*}
Similarly, we have
\begin{align*}
\label{c24}
\mathbb{P}_b\left[\Gamma_{\bm{e}_q}\right]=I_{b,a}^q+J_{b,a}^q\mathbb{P}_a\left[\Gamma_{\bm{e}_q}\right].
\stepcounter{equation}\tag{\theequation} 
\end{align*}
Thus, by solving equations (\ref{c23}) and (\ref{c24}), we obtain (\ref{c21}). Therefore, from (\ref{c6}), (\ref{c1}), (\ref{c3}), (\ref{c4}), and (\ref{c21}), we complete the calculation of the expectation $\mathbb{P}_x\left[\Gamma_{\bm{e}_q}\right]$.

Next, we consider the limit of $r_q(0)\mathbb{P}_{x}\left[\Gamma_{\bm{e}_q}\right]$ as $q\to 0+.$ By Theorem \ref{b28} and by (\ref{c3}), (\ref{b35}), and (\ref{b36}), we obtain
\begin{align*}
\label{c27}
&\lim_{q\to 0+}r_q(0)\mathbb{P}_x\left[\int_0^{T_a\wedge T_b}qe^{-qs}ds\right]\\
&=\lim_{q\to 0+}\frac{h_q(x-b)h_q(b-a)+h_q(x-a)h_q(a-b)-h_q(a-b)h_q(b-a)}{h_q^B(a-b)}\\
&=\frac{h(x-b)h(b-a)+h(x-a)h(a-b)-h(a-b)h(b-a)}{h^B(a-b)}\\
&=h(x-a)-\mathbb{P}_x(T_b<T_a)h(b-a).
\stepcounter{equation}\tag{\theequation} 
\end{align*}
By the monotone convergence theorem and by (\ref{b40}), we have
\begin{align}
\label{c30}
\lim_{q\to 0+}n^a[e^{-qT_b}]=n^a(T_b<\infty)=\frac{1}{h^B(a-b)},
\end{align}
and by (\ref{b9}),
\begin{align}
\label{c31}
\lim_{q\to 0+}\frac{r_q(a-b)}{r_q(0)}=\lim_{q\to 0+}\mathbb{P}_{b-a}[e^{-qT_0}]=\mathbb{P}_{b-a}(T_0<\infty)=1.
\end{align}
Thus, by (\ref{c6}), (\ref{b22}), (\ref{c30}), and (\ref{c31}), we have
\begin{align}
\label{c32}
\lim_{q\to 0+}r_q(0)I_{a,b}^q=\lim_{q\to 0+}\frac{1-n^a[e^{-qT_b}]h_q(b-a)}{\frac{1}{r_q(0)}+\lambda_a+n^a[e^{-qT_b}] \cdot\frac{r_q(a-b)}{r_q(0)}}=\frac{h(a-b)}{1+\lambda_ah^B(a-b)}.
\end{align}
By the monotone convergence theorem and by (\ref{b44}), we have
\begin{align*}
\label{c35}
\lim_{q\to 0+}J_{a,b}^q=\mathbb{P}_a\left[e^{-\lambda_aL_{T_b}^a}\right]=\frac{1}{h^B(a-b)}\int_0^\infty e^{-\frac{u}{h^B(a-b)}}e^{-\lambda_a u}du=\frac{1}{1+\lambda_ah^B(a-b)}.
\stepcounter{equation}\tag{\theequation} 
\end{align*}
Thus, by (\ref{c21}), (\ref{c32}), and (\ref{c35}), we obtain
\begin{align*}
\label{c36}
\lim_{q\to 0+}r_q(0)&\mathbb{P}_a\left[\Gamma_{\bm{e}_q}\right]=\lim_{q\to 0+}\frac{r_q(0)I_{a,b}^q+J_{a,b}^q\cdot r_q(0)I_{b,a}^q}{1-J_{a,b}^qJ_{b,a}^q}=\frac{1+\lambda_bh(a-b)}{\lambda_a+\lambda_b+\lambda_a\lambda_b h^B(a-b)}.
\stepcounter{equation}\tag{\theequation} 
\end{align*}
By the monotone convergence theorem and by (\ref{b44}), we have
\begin{align*}
\label{c39}
\lim_{q\to 0+}\mathbb{P}_x\left[e^{-\lambda_a L_{T_b}^a}e^{-qT_b},\ T_a<T_b\right]&=\mathbb{P}_x\left[e^{-\lambda_a L_{T_b}^a},\ T_a<T_b\right]\\
&=\mathbb{P}_x\left[e^{-\lambda_a L_{T_b}^a}\right]-\mathbb{P}_x(T_b<T_a)\\
&=\mathbb{P}_x(T_a<T_b)\cdot \frac{1}{1+\lambda_ah^B(b-a)}.
\stepcounter{equation}\tag{\theequation} 
\end{align*}
Therefore, the assertion (\ref{c42}) holds by (\ref{c1}), (\ref{c4}), (\ref{c27}), (\ref{c32}), (\ref{c36}), and (\ref{c39}). Moreover, from (\ref{b31}), we also obtain (\ref{c42-1}). 
\end{proof}

We define
\begin{align}
\label{c43}
N_{a,b,t}^{q,\lambda_a,\lambda_b}:&=\displaystyle r_q(0)\mathbb{P}_x\left[\Gamma_{\bm{e}_q};\ t<\bm{e}_q|\F_t\right],\\
\label{c44}
M_{a,b,t}^{q,\lambda_a,\lambda_b}:&=r_q(0)\mathbb{P}_x\left[\Gamma_{\bm{e}_q}|\F_t\right]
\end{align}
for $q>0.$

\begin{thm}
\label{c46}
Let $x\in \R$. Then, $(M_{a,b,t}^{(0),\lambda_a,\lambda_b},\ t\ge 0)$ is a non-negative $((\F_t),\mathbb{P}_x)$-martingale, and it holds that 
\begin{align}
\label{c47}
\lim_{q\to 0+}N_{a,b,t}^{q,\lambda_a,\lambda_b}=\lim_{q\to 0+}M_{a,b,t}^{q,\lambda_a,\lambda_b}=M_{a,b,t}^{(0),\lambda_a,\lambda_b}\ \mathrm{a.s.\ and\ in}\ L^1(\mathbb{P}_x).
\end{align}
Consequently, if $M_{a,b,0}^{(0),\lambda_a,\lambda_b}>0$ under $\mathbb{P}_x$, it holds that 
\begin{align}
\label{c48}
\lim_{q\to 0+}\frac{\mathbb{P}_x[F_s\cdot \Gamma_{\bm{e}_q}]}{\mathbb{P}_x[\Gamma_{\bm{e}_q}]}= \mathbb{P}_x\left[F_s \cdot\frac{M_{a,b,s}^{(0),\lambda_a,\lambda_b}}{M_{a,b,0}^{(0),\lambda_a,\lambda_b}}\right]
\end{align}
for all bounded $\F_s$-measurable functionals $F_s$.
\end{thm}
\begin{proof}
We denote by $\mathbb{P}_x^\mathscr{G}[\cdot]$ the conditional expectation for $\mathbb{P}_x$ with respect to a $\sigma$-algebra $\mathscr{G}$. By the lack of memory property of an exponential distribution and by the Markov property, we have
\begin{align*}
\label{c49}
N_{a,b,t}^{q,\lambda_a,\lambda_b}&=r_q(0)\mathbb{P}_x^{\F_t}\left[e^{-\lambda_a L_{\bm{e}_q-t+t}^a-\lambda_b L_{\bm{e}_q-t+t}^b}\Big|\ t<\bm{e}_q\right]\mathbb{P}_x(t<\bm{e}_q)\\
&=r_q(0)e^{-qt}\mathbb{P}_x\left[e^{-\lambda_a (L_{\bm{e}_q}^a\circ \theta_t +L_t^a)-\lambda_b(L_{\bm{e}_q}^b\circ \theta_t+L_t^b)}\Big|\F_t\right]\\
&=r_q(0)e^{-qt}e^{-\lambda_aL_t^a-\lambda_bL_t^b}\mathbb{P}_{X_t}\left[e^{-\lambda_aL_{\bm{e}_q}^a-\lambda_b L_{\bm{e}_q}^b}\right].
\stepcounter{equation}\tag{\theequation} 
\end{align*}
By (\ref{c42-1}) of Proposition \ref{c40}, we obtain
\begin{align}
\label{c51}
\lim_{q\to 0+} N_{a,b,t}^{q,\lambda_a,\lambda_b}=M_{a,b,t}^{(0),\lambda_a,\lambda_b}\ \mathrm{a.s.\ and\ in\ }L^1(\mathbb{P}_x).
\end{align}
Since $qr_q(0)\to 0$ as $q\to 0+$ by e.g., Lemma 15.5 of \cite{Tukada}, we have
\begin{align*}
\label{c53}
M_{a,b,t}^{q,\lambda_a,\lambda_b}-N_{a,b,t}^{q,\lambda_a,\lambda_b}&=r_q(0)\mathbb{P}_{X_t}\left[\Gamma_{\bm{e}_q};\ \bm{e}_q\le t|\F_t\right]\\
&=q r_q(0)\int_0^t e^{-\lambda_aL_u^a-\lambda_bL_u^b}e^{-qu}du\\
&\le qr_q(0)\to 0
\stepcounter{equation}\tag{\theequation} 
\end{align*}
as $q\to 0+$. Thus, we obtain
\begin{align}
\label{c56}
\lim_{q\to 0+}M_{a,b,t}^{q,\lambda_a,\lambda_b}=M_{a,b,t}^{(0),\lambda_a,\lambda_b}\ \mathrm{a.s.\ and\ }\mathrm{in}\ L^1(\mathbb{P}_x).
\end{align}
Therefore, the proof is complete.
\end{proof}


\section{Hitting time clock}
\label{S4}
Let us find the limit of $h^B(c)\mathbb{P}_{x}\left[\Gamma_{T_c}\right]$ as $c\to \pm \infty$.

\begin{prop}
\label{d38}
For distinct points $a,b\in \R$ and for constants $\lambda_a,\lambda_b>0$, it holds that
\begin{align}
\label{d40}
\lim_{c\to \pm \infty}h^B(c)\mathbb{P}_x\left[\Gamma_{T_c}\right]=\varphi_{a,b}^{(\pm 1),\lambda_a,\lambda_b}(x)
\end{align}
for all $x\in \R$, and 
\begin{align}
\label{d40-1}
\lim_{c\to \pm \infty}h^B(c)\mathbb{P}_{X_t}\left[\Gamma_{T_c}\right]=\varphi_{a,b}^{(\pm 1),\lambda_a,\lambda_b}(X_t)\ \mathrm{a.s.\ and\ in\ }L^1(\P_x).
\end{align}
\end{prop}

Before proving Proposition \ref{d38}, we prove the following lemmas.

\begin{lem}
\label{d2}
For $x\in \R,$ for distinct points $a,b,c\in \R,\ $and for $\lambda_a>0$, it holds that
\begin{align*}
\label{d3}
\mathbb{P}_x&\left[e^{-\lambda_aL_{T_c}^a};\ T_a<T_c<T_b\right]=\mathbb{P}_x(T_a<T_b\wedge T_c)\cdot \frac{\displaystyle \mathbb{P}_a\left[e^{-\lambda_aL_{T_c}^a}\right]-\mathbb{P}_a\left[e^{-\lambda_a L_{T_b}^a}\right]\mathbb{P}_b\left[e^{-\lambda_a L_{T_c}^a}\right]}{\displaystyle 1-\mathbb{P}_c\left[e^{-\lambda_a L_{T_b}^a}\right]\mathbb{P}_b\left[e^{-\lambda_a L_{T_c}^a}\right]}.
\stepcounter{equation}\tag{\theequation} 
\end{align*}
In particular, it can be represented only by $h$.
\end{lem}
\begin{proof}
By the strong Markov property at $T_a$, we have
\begin{align*}
\label{d4}
\mathbb{P}_x\left[e^{-\lambda_aL_{T_c}^a};\ T_a<T_c<T_b\right]
=\mathbb{P}_x(T_a<T_b\wedge T_c)\mathbb{P}_a\left[e^{-\lambda_aL_{T_c}^a};\ T_c<T_b\right].
\stepcounter{equation}\tag{\theequation} 
\end{align*}
By the strong Markov property at $T_b$, we have
\begin{align*}
\label{d5}
\mathbb{P}_a\left[e^{-\lambda_aL_{T_c}^a};\ T_b<T_c\right]&=\mathbb{P}_a\left[1_{\{T_b<T_c\}}\mathbb{P}_a\left[e^{-\lambda_a( L_{T_c-T_b+T_b}^a-L_{T_b}^a+L_{T_b}^a)}\Big|\F_{T_b}\right]\right]\\
&=\mathbb{P}_a\left[1_{\{T_b<T_c\}}e^{-\lambda_a L_{T_b}^a}\mathbb{P}_a\left[e^{-\lambda_a L_{t}^a\circ \theta_s}\Big|\F_{s}\right]_{t=T_c\circ \theta_s,\ s=T_b}\right]\\
&=\mathbb{P}_a\left[e^{-\lambda_a L_{T_b}^a};\ T_b<T_c\right]\mathbb{P}_b\left[e^{-\lambda_a L_{T_c}^a}\right].
\stepcounter{equation}\tag{\theequation} 
\end{align*}
Similarly, we have
\begin{align*}
\label{d6}
\mathbb{P}_a\left[e^{-\lambda_a L_{T_b}^a};\ T_c<T_b\right]&=\mathbb{P}_a\left[e^{-\lambda_aL_{T_c}^a};\ T_c<T_b\right]\mathbb{P}_c\left[e^{-\lambda_a L_{T_b}^a}\right].
\stepcounter{equation}\tag{\theequation}
\end{align*}
Thus, we have by (\ref{d5}) and (\ref{d6}),
\begin{align*}
\label{d7}
&\mathbb{P}_a\left[e^{-\lambda_aL_{T_c}^a};\ T_c<T_b\right]\\
&=\mathbb{P}_a\left[e^{-\lambda_aL_{T_c}^a}\right]-\mathbb{P}_a\left[e^{-\lambda_aL_{T_c}^a};\ T_b<T_c\right]\\
&=\mathbb{P}_a\left[e^{-\lambda_aL_{T_c}^a}\right]-\mathbb{P}_a\left[e^{-\lambda_a L_{T_b}^a};\ T_b<T_c\right]\mathbb{P}_b\left[e^{-\lambda_a L_{T_c}^a}\right]\\
&=\mathbb{P}_a\left[e^{-\lambda_aL_{T_c}^a}\right]-\left(\mathbb{P}_a\left[e^{-\lambda_a L_{T_b}^a}\right]-\mathbb{P}_a\left[e^{-\lambda_a L_{T_b}^a};\ T_c<T_b\right]\right)\mathbb{P}_b\left[e^{-\lambda_a L_{T_c}^a}\right]\\
&=\mathbb{P}_a\left[e^{-\lambda_aL_{T_c}^a}\right]-\mathbb{P}_a\left[e^{-\lambda_a L_{T_b}^a}\right]\mathbb{P}_b\left[e^{-\lambda_a L_{T_c}^a}\right]+\mathbb{P}_a\left[e^{-\lambda_aL_{T_c}^a};\ T_c<T_b\right]\mathbb{P}_c\left[e^{-\lambda_a L_{T_b}^a}\right]\mathbb{P}_b\left[e^{-\lambda_a L_{T_c}^a}\right].
\stepcounter{equation}\tag{\theequation}
\end{align*}
Therefore, solving the equation (\ref{d7}), we have
\begin{align}
\label{d8}
\mathbb{P}_a\left[e^{-\lambda_aL_{T_c}^a};\ T_c<T_b\right]=\frac{\displaystyle \mathbb{P}_a\left[e^{-\lambda_aL_{T_c}^a}\right]-\mathbb{P}_a\left[e^{-\lambda_a L_{T_b}^a}\right]\mathbb{P}_b\left[e^{-\lambda_a L_{T_c}^a}\right]}{\displaystyle 1-\mathbb{P}_c\left[e^{-\lambda_a L_{T_b}^a}\right]\mathbb{P}_b\left[e^{-\lambda_a L_{T_c}^a}\right]}.
\end{align}
Summarizing the above calculation, we obtain the equation (\ref{d3}).

Finally, since
\begin{align}
\label{d10}
 \mathbb{P}_a\left[e^{-\lambda_aL_{T_c}^a}\right]=\frac{1}{1+\lambda_a h^B(c-a)}
\end{align}
by (\ref{b44}) and (\ref{b36}), and since
\begin{align*}
\label{d11}
\mathbb{P}_b\left[e^{-\lambda_a L_{T_c}^a}\right]&=\mathbb{P}_{b-a}(T_0>T_{c-a})+\frac{\mathbb{P}_{b-a}(T_0<T_{c-a})}{h^B(c-a)}\int_0^\infty e^{-\frac{u}{h^B(c-a)}}e^{-\lambda_au}du\\
&=\mathbb{P}_b(T_a>T_c)+\frac{\mathbb{P}_b(T_a<T_c)}{1+\lambda_ah^B(c-a)}\\
&=\frac{1+\lambda_a\{h(a-c)+h(b-a)-h(b-c)\}}{1+\lambda_a h^B(c-a)},
\stepcounter{equation}\tag{\theequation}
\end{align*}
by (\ref{b44}) and (\ref{b36}), (\ref{d3}) can be represented only by $h$.
\end{proof}

\begin{lem}
\label{d12}
For $x\in \R,$ for distinct points $a,b,c\in \R,\ $and for constants $\lambda_a,\lambda_b>0$, it holds that
\begin{align*}
\label{d13}
&\mathbb{P}_x\left[\Gamma_{T_c};\ T_a<T_b<T_c\right]\\
&=\mathbb{P}_x(T_a<T_b\wedge T_c)\mathbb{P}_a\left[e^{-\lambda_a L_{T_b}^a};\ T_b<T_c\right]\\
&\qquad \times \frac{\displaystyle \mathbb{P}_b\left[e^{-\lambda_b L_{T_c}^b};\ T_c<T_a\right]+\mathbb{P}_b\left[e^{-\lambda_bL_{T_a}^b};\ T_a<T_c\right]\mathbb{P}_a\left[e^{-\lambda_a L_{T_c}^a};\ T_c<T_b\right]}{\displaystyle 1-\mathbb{P}_b\left[e^{-\lambda_bL_{T_a}^b};\ T_a<T_c\right]\mathbb{P}_a\left[e^{-\lambda_aL_{T_b}^a};\ T_b<T_c\right]}.
\stepcounter{equation}\tag{\theequation} 
\end{align*}
In particular, it can be represented only by $h$.
\end{lem}
\begin{proof}
By the strong Markov property at $T_a$ and $T_b$, we have 
\begin{align*}
\label{d14}
\mathbb{P}_x\left[\Gamma_{T_c};\ T_a<T_b<T_c\right]&=\mathbb{P}_x(T_a<T_b\wedge T_c)\mathbb{P}_a\left[\Gamma_{T_c};\ T_b<T_c\right]\\
&=\mathbb{P}_x(T_a<T_b\wedge T_c)\mathbb{P}_a\left[e^{-\lambda_a L_{T_b}^a};\ T_b<T_c\right]\mathbb{P}_b\left[\Gamma_{T_c}\right].
\stepcounter{equation}\tag{\theequation}
\end{align*}
By the strong Markov property, we have
\begin{align*}
\label{d15}
\mathbb{P}_b\left[\Gamma_{T_c}\right]&=\mathbb{P}_b\left[e^{-\lambda_b L_{T_c}^b};\ T_c<T_a\right]+\mathbb{P}_b\left[\Gamma_{T_c};\ T_a<T_c\right]\\
&=\mathbb{P}_b\left[e^{-\lambda_b L_{T_c}^b};\ T_c<T_a\right]+\mathbb{P}_b\left[e^{-\lambda_bL_{T_a}^b};\ T_a<T_c\right]\mathbb{P}_a\left[\Gamma_{T_c}\right]\\
&=\mathbb{P}_b\left[e^{-\lambda_b L_{T_c}^b};\ T_c<T_a\right]+\mathbb{P}_b\left[e^{-\lambda_bL_{T_a}^b};\ T_a<T_c\right]\\
&\qquad \times \left(\mathbb{P}_a\left[e^{-\lambda_a L_{T_c}^a};\ T_c<T_b\right]+\mathbb{P}_a\left[e^{-\lambda_aL_{T_b}^a};\ T_b<T_c\right]\mathbb{P}_b\left[\Gamma_{T_c}\right]\right).
\stepcounter{equation}\tag{\theequation}
\end{align*}
Thus, solving the equation (\ref{d15}), we have
\begin{align*}
\label{d16}
\mathbb{P}_b\left[\Gamma_{T_c}\right]=\frac{\displaystyle \mathbb{P}_b\left[e^{-\lambda_b L_{T_c}^b};\ T_c<T_a\right]+\mathbb{P}_b\left[e^{-\lambda_bL_{T_a}^b};\ T_a<T_c\right]\mathbb{P}_a\left[e^{-\lambda_a L_{T_c}^a};\ T_c<T_b\right]}{\displaystyle 1-\mathbb{P}_b\left[e^{-\lambda_bL_{T_a}^b};\ T_a<T_c\right]\mathbb{P}_a\left[e^{-\lambda_aL_{T_b}^a};\ T_b<T_c\right]}.
\stepcounter{equation}\tag{\theequation}
\end{align*}
Summarizing the above calculation, we obtain the equation (\ref{d13}).

Finally, from (\ref{d8}), (\ref{d10}), and (\ref{d11}), (\ref{d13}) can be represented only by $h$.
\end{proof}

Let us proceed to the proof of Proposition \ref{d38}.

\begin{proof}[Proof of Proposition \ref{d38}]
First, we have
\begin{align*}
\label{d1}
\mathbb{P}_x\left[\Gamma_{T_c}\right]&=\mathbb{P}_x(T_c<T_a\wedge T_b)\\
&\qquad +\mathbb{P}_x\left[e^{-\lambda_aL_{T_c}^a};\ T_a<T_c<T_b\right]+\mathbb{P}_x\left[e^{-\lambda_bL_{T_c}^b};\ T_b<T_c<T_a\right]\\
&\qquad +\mathbb{P}_x\left[\Gamma_{T_c};\ T_a<T_b<T_c\right]+\mathbb{P}_x\left[\Gamma_{T_c};\ T_b<T_a<T_c\right].
\stepcounter{equation}\tag{\theequation} 
\end{align*}
Therefore, from (\ref{b49}), (\ref{d3}), (\ref{d13}), and (\ref{d1}), we complete the calculation of the expectation $\mathbb{P}_x\left[\Gamma_{T_c}\right]$.

Next, we consider the limit of $h^B(c)\mathbb{P}_{x}\left[\Gamma_{T_c}\right]$ as $c\to \pm \infty$. By (\ref{b49}), we have
\begin{align*}
\label{d21}
\lim_{c\to \pm \infty}h^B(c)\mathbb{P}_{x}(T_c<T_a\wedge T_b)&=\lim_{c\to \pm \infty}h^B(c)\cdot \frac{\mathbb{P}_x(T_c<T_a)-\mathbb{P}_x(T_b<T_a)\mathbb{P}_b(T_c<T_a)}{1-\mathbb{P}_c(T_b<T_a)\mathbb{P}_b(T_c<T_a)}\\
&=h^{(\pm 1)}(x-a)-\mathbb{P}_{x}(T_b<T_a)h^{(\pm 1)}(b-a),
\stepcounter{equation}\tag{\theequation}
\end{align*}
since
\begin{align*}
\label{d22}
\lim_{c\to \pm \infty}h^B(c)\mathbb{P}_b(T_c<T_a)&=\lim_{c\to \pm \infty}\frac{\{h(a-c)-h(b-c)\}+h(b-a)}{\frac{h^B(c-a)}{h^B(c)}}=h^{(\pm 1)}(b-a)
\stepcounter{equation}\tag{\theequation}
\end{align*}
by (\ref{b36}) and (\ref{b29-1}), and
\begin{align*}
\label{d23}
\lim_{c\to \pm \infty}\mathbb{P}_c(T_b<T_a)&=\lim_{c\to \pm \infty}\frac{h(a-b)+h(c-a)-h(c-b)}{h^B(a-b)}=\frac{h^{(\mp 1)}(a-b)}{h^B(a-b)}
\stepcounter{equation}\tag{\theequation}
\end{align*}
by (\ref{b36}) and (\ref{b29-1}). By (\ref{d3}), we have
\begin{align*}
\label{d27}
&\displaystyle\lim_{c\to \pm \infty}h^B(c)\mathbb{P}_x\left[e^{-\lambda_aL_{T_c}^a};\ T_a<T_c<T_b\right]\\
&=\displaystyle\lim_{c\to \pm \infty}h^B(c)\mathbb{P}_x(T_a<T_b\wedge T_c)\cdot \frac{\displaystyle \mathbb{P}_a\left[e^{-\lambda_aL_{T_c}^a}\right]-\mathbb{P}_a\left[e^{-\lambda_a L_{T_b}^a}\right]\mathbb{P}_b\left[e^{-\lambda_a L_{T_c}^a}\right]}{\displaystyle 1-\mathbb{P}_c\left[e^{-\lambda_a L_{T_b}^a}\right]\mathbb{P}_b\left[e^{-\lambda_a L_{T_c}^a}\right]}\\
&=\mathbb{P}_x(T_a<T_b)\cdot \frac{h^{(\pm 1)}(a-b)}{1+\lambda_ah^B(b-a)},
\stepcounter{equation}\tag{\theequation}
\end{align*}
since 
\begin{align*}
\label{d28}
\lim_{c\to \pm \infty}h^B(c)\mathbb{P}_a\left[e^{-\lambda_aL_{T_c}^a}\right]=\lim_{c\to \pm \infty}\frac{1}{\frac{1}{h^B(c)}+\lambda_a \frac{h^B(c-a)}{h^B(c)}}=\frac{1}{\lambda_a}
\stepcounter{equation}\tag{\theequation}
\end{align*}
by (\ref{d10}) and (\ref{b38}),
\begin{align*}
\label{d29}
\lim_{c\to \pm \infty}h^B(c)\mathbb{P}_b\left[e^{-\lambda_a L_{T_c}^a}\right]&=\lim_{c\to \pm \infty}\frac{1+\lambda_a\{[h(a-c)-h(b-c)]+h(b-a)\}}{\frac{1}{h^B(c)}+\lambda_a \frac{h^B(c-a)}{h^B(c)}}\\
&=\frac{1}{\lambda_a}+h^{(\pm 1)}(b-a)
\stepcounter{equation}\tag{\theequation}
\end{align*}
by (\ref{d11}), (\ref{b29-1}), and (\ref{b38}), and 
\begin{align*}
\label{d30}
\lim_{c\to \pm \infty}\mathbb{P}_c\left[e^{-\lambda_a L_{T_b}^a}\right]&=\lim_{c\to \pm \infty}\frac{1+\lambda_a\{h(a-b)+[h(c-a)-h(c-b)]\}}{1+\lambda_a h^B(b-a)}\\
&=\frac{1+\lambda_a h^{(\mp 1)}(a-b)}{1+\lambda_ah^B(b-a)}
\stepcounter{equation}\tag{\theequation}
\end{align*}
by (\ref{d11}) and (\ref{b29-1}). By (\ref{d13}), we have
\begin{align*}
\label{d34}
&\lim_{c\to \pm \infty}h^B(c)\mathbb{P}_{x}\left[\Gamma_{T_c};\ T_a<T_b<T_c\right]\\
&=\lim_{c\to \pm \infty}h^B(c)\mathbb{P}_x(T_a<T_b\wedge T_c)\mathbb{P}_a\left[e^{-\lambda_a L_{T_b}^a};\ T_b<T_c\right]\\
&\qquad \times \frac{\displaystyle \mathbb{P}_b\left[e^{-\lambda_b L_{T_c}^b};\ T_c<T_a\right]+\mathbb{P}_b\left[e^{-\lambda_bL_{T_a}^b};\ T_a<T_c\right]\mathbb{P}_a\left[e^{-\lambda_a L_{T_c}^a};\ T_c<T_b\right]}{\displaystyle 1-\mathbb{P}_b\left[e^{-\lambda_bL_{T_a}^b};\ T_a<T_c\right]\mathbb{P}_a\left[e^{-\lambda_aL_{T_b}^a};\ T_b<T_c\right]}\\
&=\mathbb{P}_{x}(T_a<T_b)\cdot \frac{1}{1+\lambda_ah^B(b-a)}\cdot\frac{1+\lambda_ah^{(\pm 1)}(b-a)}{\lambda_a\lambda_b h^B(b-a)+\lambda_a+\lambda_b},
\stepcounter{equation}\tag{\theequation} 
\end{align*}
since
\begin{align*}
\label{d35}
&\lim_{c\to \pm\infty}h^B(c)\mathbb{P}_a\left[e^{-\lambda_aL_{T_c}^a};\ T_c<T_b\right]\\
&=\lim_{c\to \pm \infty}\frac{\displaystyle h^B(c)\mathbb{P}_a\left[e^{-\lambda_aL_{T_c}^a}\right]-\mathbb{P}_a\left[e^{-\lambda_a L_{T_b}^a}\right]h^B(c)\mathbb{P}_b\left[e^{-\lambda_a L_{T_c}^a}\right]}{\displaystyle 1-\mathbb{P}_c\left[e^{-\lambda_a L_{T_b}^a}\right]\mathbb{P}_b\left[e^{-\lambda_a L_{T_c}^a}\right]}=\frac{h^{(\pm 1)}(a-b)}{1+\lambda_ah^B(b-a)}
\stepcounter{equation}\tag{\theequation}
\end{align*}
by (\ref{d8}), (\ref{d29}), (\ref{d30}), and (\ref{d34}), and
\begin{align*}
\label{d36}
\lim_{c\to \pm \infty}\mathbb{P}_a\left[e^{-\lambda_aL_{T_b}^a};\ T_b<T_c\right]=\mathbb{P}_a\left[e^{-\lambda_a L_{T_b}^a}\right]=\frac{1}{1+\lambda_a h^B(b-a)}
\stepcounter{equation}\tag{\theequation}
\end{align*}
by (\ref{d10}). Therefore, the assertion (\ref{d40}) follows from (\ref{d1}), (\ref{d21}), (\ref{d27}), and (\ref{d34}).

Finally, we show the $L^1(\P_x)$-convergence. By the subadditivity of $h$, we have
\begin{align}
\label{d51}
h(a-c)+h(X_t-a)-h(X_t-c)\le h(X_t-a)+h(a-X_t).
\end{align}
The right-hand side belongs to $L^1(\mathbb{P}_x).$ See, e.g., proof of Theorem 15.2 of \cite{Tukada}. Thus, by the dominated convergence theorem and by (\ref{d22}),
\begin{align}
\label{d53}
\lim_{c\to \pm \infty}h^B(c)\mathbb{P}_{X_t}(T_c<T_a)=h^{(\pm 1)}(X_t-a)\ \mathrm{in}\ L^1(\mathbb{P}_x).
\end{align}
Therefore, from this, we also obtain (\ref{d40-1}).
\end{proof}

We define
\begin{align}
\label{d42}
N_{a,b,t}^{c,\lambda_a,\lambda_b}:&=h^B(c)\displaystyle \mathbb{P}_{x}\left[\Gamma_{T_c};\ t<T_c|\F_t\right],\\
\label{d43}
M_{a,b,t}^{c,\lambda_a,\lambda_b}:&=h^B(c)\mathbb{P}_{x}\left[\Gamma_{T_c}|\F_t\right]
\end{align}
for distinct points $a,b,c\in \R$.

\begin{thm}
\label{d45}
Let $x\in \R$. Then $(M_{a,b,t}^{(\pm 1),\lambda_a,\lambda_b},\ t\ge 0)$ is a non-negative $((\F_t),\mathbb{P}_x)$-martingale, and it holds that 
\begin{align}
\label{d46}
\lim_{c\to \pm \infty}N_{a,b,t}^{c,\lambda_a,\lambda_b}=\lim_{c\to \pm \infty}M_{a,b,t}^{c,\lambda_a,\lambda_b}=M_{a,b,t}^{(\pm 1),\lambda_a,\lambda_b}\ \mathrm{a.s.\ and\ in}\ L^1(\mathbb{P}_x).
\end{align}
Consequently, if $M_{a,b,0}^{(\pm 1),\lambda_a,\lambda_b}>0$ under $\mathbb{P}_x$, it holds that
\begin{align}
\label{d47}
\lim_{c\to \pm \infty }\frac{\mathbb{P}_x[F_t\cdot \Gamma_{T_c}]}{\mathbb{P}_x[\Gamma_{T_c}]}=\mathbb{P}_x\left[F_t \cdot\frac{M_{a,b,t}^{(\pm 1),\lambda_a,\lambda_b}}{M_{a,b,0}^{(\pm 1),\lambda_a,\lambda_b}}\right]
\end{align}
for all bounded $\F_t$-measurable functionals $F_t$.
\end{thm}
The proof is almost the same as that of Theorem \ref{c46}, based on (\ref{d40-1}) of Proposition \ref{d38} and so we omit it.


\section{Two-point Hitting time clock}
\label{S5}
For $-1\le\gamma\le1$, let us find the limit of $h^C(c,-d)\mathbb{P}_x[\Gamma_{T_c\wedge T_{-d}}]$ as $(c,d)\stackrel{(\gamma)}{\to}\infty.$

\begin{prop}
\label{e8}
For distinct points $a,b\in \R$, for constants $\lambda_a,\lambda_b>0,$ and for $-1\le \gamma\le 1$, it holds that
\begin{align}
\label{e9}
\lim_{(c,d)\stackrel{(\gamma)}{\to}\infty}h^C(c,-d)\mathbb{P}_{x}\left[\Gamma_{T_c};\ T_c<T_{-d}\right]&=\frac{1+\gamma}{2}\cdot \varphi_{a,b}^{(+1),\lambda_a,\lambda_b}(x),\\
\label{e10}
\lim_{(c,d)\stackrel{(\gamma)}{\to}\infty}h^C(c,-d)\mathbb{P}_{x}\left[\Gamma_{T_{-d}};\ T_{-d}<T_{c}\right]&=\frac{1-\gamma}{2}\cdot \varphi_{a,b}^{(-1),\lambda_a,\lambda_b}(x),\\
\label{e21}
\lim_{(c,d)\stackrel{(\gamma)}{\to}\infty}h^C(c,-d)\mathbb{P}_{x}\left[\Gamma_{T_c\wedge T_{-d}}\right]&=\varphi_{a,b}^{(\gamma),\lambda_a,\lambda_b}(x),
\end{align}
for all $x\in \R$, and
  \begin{align}
\label{e21.1}
\lim_{(c,d)\stackrel{(\gamma)}{\to}\infty}h^C(c,-d)\mathbb{P}_{X_t}\left[\Gamma_{T_c};\ T_c<T_{-d}\right]&=\frac{1+\gamma}{2}\cdot \varphi_{a,b}^{(+1),\lambda_a,\lambda_b}(X_t)\ \mathrm{a.s.\ and\ in\ }L^1(\P_x),\\
\label{e21.2}
\lim_{(c,d)\stackrel{(\gamma)}{\to}\infty}h^C(c,-d)\mathbb{P}_{X_t}\left[\Gamma_{T_{-d}};\ T_{-d}<T_{c}\right]&=\frac{1-\gamma}{2}\cdot \varphi_{a,b}^{(-1),\lambda_a,\lambda_b}(X_t)\ \mathrm{a.s.\ and\ in\ }L^1(\P_x),\\
\label{e21.3}
\lim_{(c,d)\stackrel{(\gamma)}{\to}\infty}h^C(c,-d)\mathbb{P}_{X_t}\left[\Gamma_{T_c\wedge T_{-d}}\right]&=\varphi_{a,b}^{(\gamma),\lambda_a,\lambda_b}(X_t)\ \mathrm{a.s.\ and\ in\ }L^1(\P_x).
\end{align}
\end{prop}
\begin{proof}
First, it hold that for distinct points $a,b,c,d\in \R$ with $c,d>0$ and for constants $\lambda_a,\lambda_b>0$,
\begin{align*}
\label{e4}
\mathbb{P}_x&\left[\Gamma_{T_c};\ T_c<T_{-d}\right]=\frac{\mathbb{P}_x\left[\Gamma_{T_c}\right]-\mathbb{P}_x\left[\Gamma_{T_{-d}}\right]\mathbb{P}_{-d}\left[\Gamma_{T_c}\right]}{1-\mathbb{P}_c\left[\Gamma_{T_{-d}}\right]\mathbb{P}_{-d}\left[\Gamma_{T_c}\right]},
\stepcounter{equation}\tag{\theequation}
\end{align*}
since
\begin{align*}
\label{e5}
\mathbb{P}_x\left[\Gamma_{T_c};\ T_{-d}<T_c\right]&=\mathbb{P}_x\left[\Gamma_{T_{-d}};\ T_{-d}<T_c\right]\mathbb{P}_{-d}\left[\Gamma_{T_c}\right]\\
&=\left\{\mathbb{P}_x\left[\Gamma_{T_{-d}}\right]-\mathbb{P}_x\left[\Gamma_{T_{-d}};\ T_c<T_{-d}\right]\right\}\mathbb{P}_{-d}\left[\Gamma_{T_c}\right]\\
&=\mathbb{P}_x\left[\Gamma_{T_{-d}}\right]\mathbb{P}_{-d}\left[\Gamma_{T_c}\right]-\mathbb{P}_x\left[\Gamma_{T_c};\ T_c<T_{-d}\right]\mathbb{P}_c\left[\Gamma_{T_{-d}}\right]\mathbb{P}_{-d}\left[\Gamma_{T_c}\right],
\stepcounter{equation}\tag{\theequation}
\end{align*}
by the strong Markov property at $T_{-d}$ and $T_c$.

Note that if $(c,d)\stackrel{(\gamma)}{\to}\infty$, then $ \frac{d}{c}\to \frac{1+\gamma}{1-\gamma}$. Thus, we have by (\ref{b29-0}),
\begin{align}
\label{e11}
\frac{h(c-a)}{h(-c-d)}=\frac{h(c-a)}{|c-a|}\cdot\frac{|-c-d|}{h(-c-d)}\cdot \frac{|1-\frac{a}{c}|}{|-1-\frac{d}{c}|}\to \frac{1-\gamma}{2},\\
\label{e12}
\frac{h(-d)}{h(c+d)}=\frac{h(-d)}{|-d|}\cdot\frac{|-c-d|}{h(-c-d)}\cdot \frac{|-\frac{d}{c}|}{|-1-\frac{d}{c}|}\to \frac{1+\gamma}{2}
\end{align}
as $(c,d)\stackrel{(\gamma)}{\to}\infty.$ Thus, we have by (\ref{b46}) and (\ref{b34}),
\begin{align}
\label{e13}
\frac{h^B(c)}{h^C(c,-d)}\to \frac{2}{1+\gamma},\\
\label{e13-1}
\frac{h^B(d)}{h^C(c,-d)}\to \frac{2}{1-\gamma},
\end{align}
as $(c,d)\stackrel{(\gamma)}{\to}\infty.$ By (\ref{d1}), we have 
\begin{align}
\label{e17}
\lim_{(c,d)\stackrel{(\gamma)}{\to}\infty}\mathbb{P}_{-d}\left[\Gamma_{T_c}\right]=0,
\end{align}
since
\begin{align*}
\label{e14}
\lim_{(c,d)\stackrel{(\gamma)}{\to}\infty}\mathbb{P}_{-d}(T_c<T_a\wedge T_b)&=\lim_{(c,d)\stackrel{(\gamma)}{\to}\infty}\frac{\mathbb{P}_{-d}(T_c<T_a)-\mathbb{P}_{-d}(T_b<T_a)\mathbb{P}_b(T_c<T_a)}{1-\mathbb{P}_c(T_b<T_a)\mathbb{P}_b(T_c<T_a)}\\
&=\lim_{(c,d)\stackrel{(\gamma)}{\to}\infty} \mathbb{P}_{-d}(T_c<T_a)\\
&=\lim_{(c,d)\stackrel{(\gamma)}{\to}\infty}\frac{\frac{h(a-c)}{h(-c-d)}+\frac{h(-d-a)}{h(-c-d)}-1}{\frac{h(c-a)}{h(-c-d)}+\frac{h(a-c)}{h(-c-d)}}=0
\stepcounter{equation}\tag{\theequation}
\end{align*}
by (\ref{b49}), (\ref{b36}), (\ref{e11}), and (\ref{e12}),
\begin{align*}
\label{e15}
&\lim_{(c,d)\stackrel{(\gamma)}{\to}\infty}\mathbb{P}_{-d}\left[e^{-\lambda_aL_{T_c}^a};\ T_a<T_c<T_b\right]\\
&=\lim_{(c,d)\stackrel{(\gamma)}{\to}\infty}\mathbb{P}_{-d}(T_a<T_b\wedge T_c)\cdot\frac{\displaystyle \mathbb{P}_a\left[e^{-\lambda_aL_{T_c}^a}\right]-\mathbb{P}_a\left[e^{-\lambda_a L_{T_b}^a}\right]\mathbb{P}_b\left[e^{-\lambda_a L_{T_c}^a}\right]}{\displaystyle 1-\mathbb{P}_c\left[e^{-\lambda_a L_{T_b}^a}\right]\mathbb{P}_b\left[e^{-\lambda_a L_{T_c}^a}\right]}=0
\stepcounter{equation}\tag{\theequation}
\end{align*}
by (\ref{d3}) and (\ref{d30}), and
\begin{align*}
\label{e16}
&\lim_{(c,d)\stackrel{(\gamma)}{\to}\infty}\mathbb{P}_{-d}\left[\Gamma_{T_c};\ T_a<T_b<T_c\right]\\
&=\lim_{(c,d)\stackrel{(\gamma)}{\to}\infty}\mathbb{P}_{-d}(T_a<T_b\wedge T_c)\mathbb{P}_a\left[e^{-\lambda_a L_{T_b}^a};\ T_b<T_c\right]\\
&\qquad \times \frac{\displaystyle \mathbb{P}_b\left[e^{-\lambda_b L_{T_c}^b};\ T_c<T_a\right]+\mathbb{P}_b\left[e^{-\lambda_bL_{T_a}^b};\ T_a<T_c\right]\mathbb{P}_a\left[e^{-\lambda_a L_{T_c}^a};\ T_c<T_b\right]}{\displaystyle 1-\mathbb{P}_b\left[e^{-\lambda_bL_{T_a}^b};\ T_a<T_c\right]\mathbb{P}_a\left[e^{-\lambda_aL_{T_b}^a};\ T_b<T_c\right]}=0
\stepcounter{equation}\tag{\theequation}
\end{align*}
by (\ref{d13}), (\ref{b38}), and (\ref{d35}). Summarizing the above calculation, we obtain by (\ref{e4}) and (\ref{d40}),
\begin{align*}
\label{e18}
&\lim_{(c,d)\stackrel{(\gamma)}{\to}\infty}h^C(c,-d)\mathbb{P}_{x}\left[\Gamma_{T_c};\ T_c<T_{-d}\right]\\
&=\lim_{(c,d)\stackrel{(\gamma)}{\to}\infty}\frac{\frac{h^C(c,-d)}{h^B(c)}\cdot h^B(c)\mathbb{P}_{x}\left[\Gamma_{T_c}\right]-\frac{h^C(c,-d)}{h^B(-d)}\cdot  h^B(-d)\mathbb{P}_x\left[\Gamma_{T_{-d}}\right]\mathbb{P}_{-d}\left[\Gamma_{T_c}\right]}{1-\mathbb{P}_c\left[\Gamma_{T_{-d}}\right]\mathbb{P}_{-d}\left[\Gamma_{T_c}\right]}\\
&=\frac{1+\gamma}{2}\lim_{(c,d)\stackrel{(\gamma)}{\to}\infty} h^B(c)\mathbb{P}_{x}\left[\Gamma_{T_c}\right]=\frac{1+\gamma}{2}\cdot \varphi_{a,b}^{(+1),\lambda_a,\lambda_b}(x).
\stepcounter{equation}\tag{\theequation}
\end{align*}
Hence, we have proved the first equation. The second equation (\ref{e10}) can be shown similarly. The third equation (\ref{e21}) can be proved by summing up the first and the second equations. The remaining equations are clear from (\ref{d40-1}) of Proposition \ref{d38}.
\end{proof}

We define
\begin{align}
\label{e23}
N_{a,b,t}^{1,c,d,\lambda_a,\lambda_b}&:=h^C(c,-d)\mathbb{P}_{x}\left[\Gamma_{T_c\wedge T_{-d}};\ t<T_{c}\wedge T_{-d}|\F_{t}\right],\\
N_{a,b,t}^{2,c,d,\lambda_a,\lambda_b}&:=h^C(c,-d)\mathbb{P}_{x}\left[\Gamma_{T_c};\ t<T_{c}< T_{-d}|\F_{t}\right],\\
N_{a,b,t}^{3,c,d,\lambda_a,\lambda_b}&:=h^C(c,-d)\mathbb{P}_{x}\left[\Gamma_{T_{-d}};\ t<T_{-d}<T_{c}|\F_{t}\right],\\
\label{e24}
M_{a,b,t}^{1,c,d,\lambda_a,\lambda_b}&:=h^C(c,-d)\mathbb{P}_{x}\left[\Gamma_{T_c\wedge T_{-d}}|\F_{t}\right],\\
M_{a,b,t}^{2,c,d,\lambda_a,\lambda_b}&:=h^C(c,-d)\mathbb{P}_{x}\left[\Gamma_{T_c};\ T_c<T_{-d}|\F_{t}\right],\\
M_{a,b,t}^{3,c,d,\lambda_a,\lambda_b}&:=h^C(c,-d)\mathbb{P}_{x}\left[\Gamma_{T_{-d}};\ T_{-d}<T_c|\F_{t}\right]
\end{align}
for $c,d>0$.

\begin{thm}
\label{e26}
Let $x\in \R$ and $-1\le \gamma\le 1.$ Then, $(M_{a,b,t}^{(\gamma),\lambda_a,\lambda_b},\ t\ge 0)$ is a non-negative $((\F_t),\mathbb{P}_x)$-martingale, and it holds that 
\begin{align}
\label{e27}
\lim_{(c,d)\stackrel{(\gamma)}{\to}\infty}N_{a,b,t}^{1,c,d,\lambda_a,\lambda_b}=\lim_{(c,d)\stackrel{(\gamma)}{\to}\infty}M_{a,b,t}^{1,c,d,\lambda_a,\lambda_b}&=M_{a,b,t}^{(\gamma),\lambda_a,\lambda_b}\ \mathrm{a.s.\ and\ in}\ L^1(\mathbb{P}_x),\\
\lim_{(c,d)\stackrel{(\gamma)}{\to}\infty}N_{a,b,t}^{2,c,d,\lambda_a,\lambda_b}=\lim_{(c,d)\stackrel{(\gamma)}{\to}\infty}M_{a,b,t}^{2,c,d,\lambda_a,\lambda_b}&=\frac{1+\gamma}{2}\cdot M_{a,b,t}^{(+1),\lambda_a,\lambda_b}\ \mathrm{a.s.\ and\ in}\ L^1(\mathbb{P}_x),\\
\lim_{(c,d)\stackrel{(\gamma)}{\to}\infty}N_{a,b,t}^{3,c,d,\lambda_a,\lambda_b}=\lim_{(c,d)\stackrel{(\gamma)}{\to}\infty}M_{a,b,t}^{3,c,d,\lambda_a,\lambda_b}&=\frac{1-\gamma}{2}\cdot M_{a,b,t}^{(-1),\lambda_a,\lambda_b}\ \mathrm{a.s.\ and\ in}\ L^1(\mathbb{P}_x).
\end{align}
Consequently, if $M_{a,b,0}^{(\gamma),\lambda_a,\lambda_b}>0$ under $\mathbb{P}_x$, it holds that 
\begin{align}
\label{e28}
\displaystyle \lim_{(c,d)\stackrel{(\gamma)}{\to}\infty}\frac{\mathbb{P}_x[F_t\cdot \Gamma_{T_c\wedge T_{-d}}]}{\mathbb{P}_x[\Gamma_{T_c\wedge T_{-d}}]}= \mathbb{P}_x\left[F_t \cdot\frac{M_{a,b,t}^{(\gamma),\lambda_a,\lambda_b}}{M_{a,b,0}^{(\gamma),\lambda_a,\lambda_b}}\right]
\end{align}
for all bounded $\F_t$-measurable functionals $F_t$.
\end{thm}
The proof is almost the same as that of Theorem \ref{c46}, based on Proposition \ref{e8} and so we omit it.


\section{Inverse local time clock}
\label{S6}
Let us find the limit of $h^B(c)\mathbb{P}_{x}\left[\Gamma_{\eta_u^c}\right]$ as $c\to \pm \infty$.

\begin{prop}
\label{f1-1}
For distinct points $a,b\in \R$, for constants $\lambda_a,\lambda_b>0$, and for $u>0$, it holds that
\begin{align}
\label{f1-2}
\lim_{c\to \pm \infty}h^B(c)\P_x[\Gamma_{\eta_u^c}]=\varphi_{a,b}^{(\pm 1),\lambda_a,\lambda_b}(x)
\end{align}
for all $x\in \R$, and
\begin{align}
\label{f1-3}
\lim_{c\to \pm \infty}h^B(c)\P_{X_t}[\Gamma_{\eta_u^c}]=\varphi_{a,b}^{(\pm 1),\lambda_a,\lambda_b}(X_t)\ \mathrm{a.s.\ and\ in}\ L^1(\P_x).
\end{align}
\end{prop}

Before proving Proposition \ref{f1-1}, we prove the following lemma.

\begin{lem}
\label{f1}
For distinct points $a,b\in \R$, for $u>0$, and for constants $\lambda_a,\lambda_b>0$, it holds that
\begin{align}
\label{f3}
\lim_{c\to \pm \infty}\mathbb{P}_c\left[\Gamma_{\eta_u^c}\right]=1.
\end{align}
\end{lem}
\begin{proof}
Since $L^a$ may increase during excursion away from $c$, we can write
\begin{align*}
\label{f4}
L_{\eta_u^c}^a=\sum_{v\le u}(L_{\eta_v^c}^a-L_{\eta_{v-}^c}^a)=\sum_{v\le u}L_{T_c}^a(\epsilon^c(v)).
\stepcounter{equation}\tag{\theequation}
\end{align*}
Thus, using Theorem 2.7 of \cite{Kyp}, we have
\begin{align*}
\label{f6}
\mathbb{P}_c\left[\Gamma_{\eta_u^c}\right]&=\mathbb{P}_c\left[\exp \left\{-\sum_{v\le u}\left(\lambda_aL_{T_c}^a(\epsilon^c(v))+\lambda_bL_{T_c}^b(\epsilon^c(v))\right)\right\}\right]\\
&=\exp \left\{-\int_{(0,u]\times \mathcal{D}}\left(1-e^{-\lambda_a L_{T_c}^a(e)-\lambda_bL_{T_c}^{b}(e)}\right)dt\otimes n^c(de)\right\}\\
&=\exp \left\{-un^c\left[1-e^{-\lambda_a L_{T_c}^a-\lambda_bL_{T_c}^{b}}\right]\right\}\\
&=\exp\left\{-un^c\left[1-e^{-\lambda_a L_{T_c}^a};\ T_a<T_c<T_b\right]\right\}\\
&\qquad \times \exp\left\{-un^c\left[1-e^{-\lambda_b L_{T_c}^b};\ T_b<T_c<T_a\right]\right\}\\
&\qquad \times \exp\left\{-un^c\left[1-e^{-\lambda_a L_{T_c}^a-\lambda_bL_{T_c}^{b}};\ T_a<T_b<T_c\right]\right\}\\
&\qquad \times \exp\left\{-un^c\left[1-e^{-\lambda_a L_{T_c}^a-\lambda_bL_{T_c}^{b}};\ T_b<T_a<T_c\right]\right\}.
\stepcounter{equation}\tag{\theequation}
\end{align*}
By the strong Markov property for $n^c$, and by (\ref{b38}) and (\ref{b40}), we have
\begin{align*}
\label{f7}
&n^c\left[1-e^{-\lambda_a L_{T_c}^a};\ T_a<T_c<T_b\right]\\
&=\int n^c(1_{\{T_a<T_b\wedge T_c\}};\ X_{T_a}\in dx)\mathbb{P}_x^c[1-e^{-\lambda_aL_{T_c}^a},\ T_c<T_b]\\
&=n^c(T_a<T_b\wedge T_c)\mathbb{P}_x[1-e^{-\lambda_aL_{T_c}^a},\ T_c<T_b]\\
&\le n^c(T_a<T_c)\mathbb{P}_x[1-e^{-\lambda_aL_{T_c}^a},\ T_c<T_b]\to 0
\stepcounter{equation}\tag{\theequation}
\end{align*}
as $c\to \pm \infty$. Moreover, by the strong Markov property for $n^a$ and $\mathbb{P}_a$, we have
\begin{align*}
\label{f8}
&n^c\left[1-e^{-\lambda_a L_{T_c}^a-\lambda_bL_{T_c}^{b}};\ T_a<T_b<T_c\right]\\
&=n^c(T_a<T_b\wedge T_c)\left\{\mathbb{P}_a(T_b<T_c)-\mathbb{P}_a\left[\Gamma_{T_c};\ T_b<T_c\right]\right\}\\
&=n^c(T_a<T_b\wedge T_c)\left\{\mathbb{P}_a(T_b<T_c)-\mathbb{P}_a\left[e^{-\lambda_aL_{T_b}^a};\ T_b<T_c\right]\mathbb{P}_b\left[\Gamma_{T_c}\right]\right\}\to 0
\stepcounter{equation}\tag{\theequation}
\end{align*}
as $c\to \pm \infty.$ Therefore, (\ref{f3}) follows (\ref{f6}), (\ref{f7}), and (\ref{f8}). 
\end{proof}

Let us proceed to the proof of Proposition \ref{f1-1}.

\begin{proof}
It holds that for distinct points $a,b,c\in \R$,
\begin{align}
\label{f10}
\mathbb{P}_x\left[\Gamma_{\eta_u^c}\right]=\P_c\left[\Gamma_{\eta_{u}^c}\right]\cdot \mathbb{P}_x\left[\Gamma_{T_c}\right].
\end{align}
Indeed, we divide the expectation into the following events: 
\begin{align}
\label{f11}
\{T_c<T_a\wedge T_b\}\cup\{T_a<T_c<T_b\}\cup\{T_b<T_c<T_a\}\cup\{T_a<T_b<T_c\}\cup\{T_b<T_a<T_c\}.
\end{align}
We calculate the expectation of each term.

First, by the strong Markov property at $T_c$, we have
\begin{align*}
\label{f12}
\mathbb{P}_x&\left[\Gamma_{\eta_u^c};\ T_c<T_a\wedge T_b\right]=\mathbb{P}_x(T_c<T_a\wedge T_b)\mathbb{P}_c\left[\Gamma_{\eta_u^c}\right].
\stepcounter{equation}\tag{\theequation}
\end{align*}

Next, by the strong Markov property at $T_a$ and $T_c$, we have
\begin{align*}
\label{f13}
\mathbb{P}_x\left[\Gamma_{\eta_u^c};\ T_a<T_c<T_b\right]&=\mathbb{P}_x(T_a<T_b\wedge T_c)\mathbb{P}_a\left[\Gamma_{\eta_u^c};\ T_c<T_b\right]\\
&=\mathbb{P}_x(T_a<T_b\wedge T_c)\mathbb{P}_a\left[e^{-\lambda_aL_{T_c}^a};\ T_c<T_b\right]\mathbb{P}_c\left[\Gamma_{\eta_u^c}\right].
\stepcounter{equation}\tag{\theequation}
\end{align*}

Finally, by the strong Markov property, we have
\begin{align*}
\label{f14}
\mathbb{P}_x\left[\Gamma_{\eta_u^c};\ T_a<T_b<T_c\right]&=\mathbb{P}_x(T_a<T_b\wedge T_c)\mathbb{P}_a\left[\Gamma_{\eta_u^c};\ T_b<T_c\right]\\
&=\mathbb{P}_x(T_a<T_b\wedge T_c)\mathbb{P}_a\left[e^{-\lambda_aL_{T_b}^a};\ T_b<T_c\right]\mathbb{P}_b\left[\Gamma_{\eta_u^c}\right].
\stepcounter{equation}\tag{\theequation}
\end{align*}
Now, we have
\begin{align*}
\label{f15}
\mathbb{P}_b\left[\Gamma_{\eta_u^c}\right]&=\mathbb{P}_b\left[\Gamma_{\eta_u^c};\ T_a<T_c\right]+\mathbb{P}_b\left[\Gamma_{\eta_u^c};\ T_c<T_a\right]\\
&=\mathbb{P}_b\left[e^{-\lambda_bL_{T_a}^b};\ T_a<T_c\right]\mathbb{P}_a\left[\Gamma_{\eta_u^c}\right]+\mathbb{P}_b\left[e^{-\lambda_bL_{T_c}^b};\ T_c<T_a\right]\mathbb{P}_c\left[\Gamma_{\eta_u^c}\right].
\stepcounter{equation}\tag{\theequation}
\end{align*}
Similarly, we have
\begin{align*}
\label{f16}
\mathbb{P}_a\left[\Gamma_{\eta_u^c}\right]=&\mathbb{P}_a\left[e^{-\lambda_aL_{T_b}^a};\ T_b<T_c\right]\mathbb{P}_b\left[\Gamma_{\eta_u^c}\right]+\mathbb{P}_a\left[e^{-\lambda_aL_{T_c}^a};\ T_c<T_b\right]\mathbb{P}_c\left[\Gamma_{\eta_u^c}\right].
\stepcounter{equation}\tag{\theequation}
\end{align*}
Thus, by solving the simultaneous equations (\ref{f15}) and (\ref{f16}), we have
\begin{align*}
\label{f17}
\mathbb{P}_b\left[\Gamma_{\eta_u^c}\right]=\frac{\mathbb{P}_c\left[\Gamma_{\eta_u^c}\right]\left\{\mathbb{P}_b\left[e^{-\lambda_bL_{T_a}^b};\ T_a<T_c\right]\mathbb{P}_a\left[e^{-\lambda_aL_{T_c}^a};\ T_c<T_b\right]+\mathbb{P}_b\left[e^{-\lambda_bL_{T_c}^b};\ T_c<T_a\right]\right\}}{1-\mathbb{P}_b\left[e^{-\lambda_bL_{T_a}^b};\ T_a<T_c\right]\mathbb{P}_a\left[e^{-\lambda_aL_{T_b}^a};\ T_b<T_c\right]}.
\stepcounter{equation}\tag{\theequation}
\end{align*}
Summarizing the above calculation, we obtain
\begin{align*}
\label{f18}
\mathbb{P}_x&\left[\Gamma_{\eta_u^c};\ T_a<T_b<T_c\right]=\mathbb{P}_c\left[\Gamma_{\eta_u^c}\right]\mathbb{P}_x\left[\Gamma_{T_c};\ T_a<T_b<T_c\right].
\stepcounter{equation}\tag{\theequation}
\end{align*}
Summing up the above discussion (\ref{f12}), (\ref{f13}), and (\ref{f18}), we obtain the equation (\ref{f10}).

Therefore, (\ref{f1-2}) follows from (\ref{d40}) and (\ref{f10}). The remaining equations are clear from (\ref{d40-1}) of Proposition \ref{d38}.
\end{proof}

We define
\begin{align}
\label{f22}
N_{a,b,t}^{c,u,\lambda_a,\lambda_b}:&=h^B(c)\mathbb{P}_x\left[\Gamma_{\eta_u^c};\ t<\eta_{u}^c|\F_{t}\right],\\
\label{f23}
M_{a,b,t}^{c,u,\lambda_a,\lambda_b}:&=h^B(c)\mathbb{P}_{x}\left[\Gamma_{\eta_u^c}|\F_{t}\right]
\end{align}
for $c\in \R$ and $u>0$.

\begin{thm}{\label{f24}}
Let $x\in \R$. Then, it holds that 
\begin{align}
\label{f25}
\lim_{c\to \pm\infty}N_{a,b,t}^{c,u,\lambda_a,\lambda_b}=\lim_{c\to \pm\infty}M_{a,b,t}^{c,u,\lambda_a,\lambda_b}=M_{a,b,t}^{(\pm 1),\lambda_a,\lambda_b}\ \mathrm{a.s.\ and\ in}\ L^1(\mathbb{P}_x).
\end{align}
Consequently, if $M_{a,b,0}^{(\pm 1),\lambda_a,\lambda_b}>0$ under $\mathbb{P}_x$, it holds that 
\begin{align}
\label{f26}
\lim_{c\to \pm \infty}\frac{\mathbb{P}_x\left[F_t\cdot \Gamma_{\eta_u^c}\right]}{\mathbb{P}_x\left[\Gamma_{\eta_u^c}\right]}=\mathbb{P}_x\left[F_t \cdot\frac{M_{a,b,t}^{(\pm 1),\lambda_a,\lambda_b}}{M_{a,b,0}^{(\pm 1),\lambda_a,\lambda_b}}\right]
\end{align}
for all bounded $\F_t$-measurable functionals $F_t$.
\end{thm}
The proof is almost the same as that of Theorem \ref{c46}, based on Proposition \ref{f1-1} and so we omit it.

\section*{Acknowledgements} 
This research was supported by ISM. 
The research of K. Iba and K. Yano was supported by 
JSPS Open Partnership Joint Research Projects grant no. JPJSBP120209921. 
The research of K. Yano was supported by 
JSPS KAKENHI grant no.'s 19H01791, 19K21834 and 21H01002.

\bibliographystyle{plain}

\end{document}